\documentclass[10pt]{amsart}
\usepackage{times,amsmath,amsbsy,amssymb,amscd,mathrsfs}
\usepackage{slashbox}\usepackage{graphicx,subfigure,epstopdf,wrapfig,chemarrow}
\usepackage{multicol,multirow}
\usepackage{mathtools}
\usepackage[usenames,dvipsnames,svgnames,table]{xcolor}

\usepackage{comment,enumerate,multicol,xspace}

  \newcounter{mnote}
  \let\oldmarginpar\marginpar
    \renewcommand\marginpar[1]{\-\oldmarginpar[\raggedleft\footnotesize #1]%
    {\raggedright\footnotesize #1}}

\newtheorem{theorem}{Theorem}[section]
\newtheorem{lemma}[theorem]{Lemma}

\newtheorem{remark}[theorem]{Remark}

\newcommand{\dx}{\,{\rm d}x}
\newcommand{\dd}{\,{\rm d}}

\newcommand{\mbb}{\mathbb}

\newcommand{\mcal}{\mathcal}

\newcommand{\supp}{\operatorname{supp}}

\newcommand{\curl}{{\rm curl\,}}
\renewcommand{\div}{\operatorname{div}}
\newcommand{\grad}{{\rm grad\,}}

\begin{document}
\title[Multigrid Methods for Saddle Point Problems]{Multigrid Methods for Constrained Minimization Problems and Application to Saddle Point Problems}
\author[L.~Chen]{Long Chen}
\address[L.~Chen]{Department of Mathematics, University of California at Irvine, Irvine, CA 92697, USA}
\email{chenlong@math.uci.edu}
\thanks{LC has been supported by NSF Grant DMS-1418934.}


\date{\today}


\begin{abstract}
The first order condition of the constrained minimization problem leads to a saddle point problem. A multigrid method using a multiplicative Schwarz smoother for saddle point problems can thus be interpreted as a successive subspace optimization method based on a multilevel decomposition of the constraint space. Convergence theory is developed for successive subspace optimization methods based on two assumptions on the space decomposition: stable decomposition and strengthened Cauchy-Schwarz inequality, and successfully applied to the saddle point systems arising from mixed finite element methods for Poisson and Stokes equations. Uniform convergence is obtained without the full regularity assumption of the underlying partial differential equations. As a byproduct, a V-cycle multigrid method for non-conforming finite elements is developed and proved to be uniform convergent with even one smoothing step.
\end{abstract}

\maketitle

\section{Introduction}
Given a quadratic energy $E(v)$ defined on a Hilbert space $\mcal V$, we consider the constrained minimization problem:
\begin{equation}\label{intro:main-opt}
\min _{v\in \mcal K} E(v),
\end{equation}
where $\mcal K\subset \mcal V$ is the null space of  a linear and bounded operator $B$ defined on $\mcal V$. By introducing the Lagrange multiplier for the constraint, we can find the minimizer of \eqref{intro:main-opt} by solving a saddle point system. In this paper, we shall design and analyze multigrid methods for the constrained minimization problem \eqref{intro:main-opt} and apply them to the saddle point systems arisings from mixed finite element discretization of Poisson, Darcy, and Stokes equations. 

We shall adapt the constraint decomposition methods developed by Tai for nonlinear variational inequalities~\cite{Tai.X2003} to the constrained minimization problem. Let $\mcal K = \sum_{i=1}^N \mcal K_i$ be a space decomposition. Our method consists of solving a local constrained minimization problem in each subspace $\mcal K_i$ which is equivalent to solving a small saddle point problem. Thus our relaxation can be interpreted as a multiplicative overlapping Schwarz method which  is known as Vanka smoother~\cite{Vanka1986} in the context of computational fluid dynamics. With a proper multilevel decomposition, our method becomes the classical V-cycle multigrid method.

Assuming that the decomposition $\mcal K = \sum_{i=1}^N \mcal K_i$ satisfies  two assumptions: energy stable decomposition (SD) and strengthened Cauchy-Schwarz inequality (SCS), we are able to prove the convergence of our method
$$
E(u^{k+1}) - E(u)\leq \left (1-\frac{1}{1+C_AC_S}\right) \left [E(u^k) - E(u)\right ],
$$
where $u^{k}$ is the $k$-th iteration, and $C_A$ and $C_S$ are positive constants in (SD) and (SCS).  We also extend the analysis to the case where the local constrained minimization problem is not solved exactly but one gradient iteration is applied. 

It is known that numerically multiplicative Schwarz smoother leads to an efficient multigrid methods for saddle point problems~\cite{Schoberl1999,Schoberl2003}, however, theoretical analysis for the convergence is only available for less efficient additive versions~\cite{Schoberl1999,Schoberl2003}. Our new framework can fill this gap. Furthermore, the optimal choice of the relaxation parameter used in the inexact solvers of local problems can be derived from the minimization point of view.


We then apply our method to the saddle point systems arising from mixed finite element methods of Poisson, Darcy, and Stokes equations. By verifying assumptions (SD) and (SCS) for multilevel decompositions of H(div) element spaces, we will prove the uniform convergence of a V-cycle multigrid method for mixed finite element methods for the Poisson and Darcy equations. 
Our smoother is related to the overlapping Schwarz method developed for H(div) problems in~\cite{Ewing.R;Wang.J1992a,Vassilevski.P;Wang.J1992,Mathew1993,Mathew1993a,Arnold.D;Falk.R;Winther.R2000}. But our analysis from the energy minimization point of view is more transparent. 
We note that a similar stable multilevel decomposition for the Raviart-Thomas space has been proposed in~\cite{Vassilevski.P;Wang.J1992} in two dimensions and in~\cite{Hiptmair1999,Arnold.D;Falk.R;Winther.R2000} in three dimensions. Our decomposition for three dimensional case is new and does not require the duality argument and thus relax the full regularity assumption needed in~\cite{Hiptmair1999,Arnold.D;Falk.R;Winther.R2000}. 

We use the equivalence between Crouzeix-Raviart (CR) non-conforming methods and mixed methods to develop a V-cycle multigrid method for non-conforming methods of Poisson equation and prove its uniform convergence. Existing convergence proofs of multigrid methods for non-conforming methods~\cite{Brenner.S1989,Oswald.P1997,Brenner.S1999,Brenner2003a,Oswald.P2008} cannot cover V-cycles with few smoothing steps while our new framework can. The two ingredients of our new multigrid method for non-conforming methods are: the overlapping Schwarz smoothers, and inter-grid transfer operators through the nested flux spaces. 

For discrete Stokes equations, we apply our theory to divergence free and nested finite element spaces, e.g., Scott-Voligious elements~\cite{Scott.L;Vogelius.M1985}. Again traditional multigrid convergence proofs for Stokes equations requires the full regularity assumption~\cite{Verfurth1984,Brenner.S1990,Brenner1996a,Braess.D;Sarazin.R1997,Zulehner.W2000,Olshanskii2011}. Using the framework developed in this paper, we can obtain multigrid convergence without the full regularity assumption. Very recently, Brenner, Li, and Sung~\cite{Brenner2014} have developed new multigrid methods for Stokes equations and have proved the uniform convergence without the full regularity assumption. The convergence result of~\cite{Brenner2014} is, however, restricted to W-cycle multigrid methods with sufficient many smoothing steps. Here we consider V-cycle multigrid with only one smoothing. 
Furthermore, smoothers developed in~\cite{Brenner2014} are less efficient than Vanka-type smoothers considered here; see numerical examples in~\cite{Schoberl1999,Brenner2014}. On the other hand, the framework developed in~\cite{Brenner2014} can be applied to any stable mixed finite element discretization of Stokes equation and in \cite{Brenner2015} such convergence theory is also extended to the Darcy systems, while the current theory can be only applied to the case when the constrained subspaces are nested. For non-nested constrained subspaces, an additional projector is needed and an analysis for W-cycle multigrid without the full regularity assumption can be found in~\cite{Chen2015b}. For popular finite element pairs of Stokes equations, a fast multigrid method using least square distributive Gauss-Sedel smoother has been developed in~\cite{Wang2013} for Stokes equations and generalize to Oseen problem in~\cite{Chen2013e}.

Although most of the abstract theory, either based on the Xu-Zikatanov identity~ \cite{Xu.J;Zikatanov.L2002} or following the Tai-Xu approach~\cite{Tai2001}, has been developed in certain form in the literature, the application to multigrid methods for solving saddle point systems are new and lead to several new contribution of the multigrid theory for saddle point systems: a convergence proof of V-cycle with even one smoothing step, a convergence proof without full regularity assumption, and a convergence proof for the multiplicative Schwarz smoother. Stable decomposition of several finite element spaces established in this paper also have their own interest. 

The rest of this paper is structured as follows. In Section 2, we introduce the algorithm. In Section 3, we give a convergence proof using the X-Z identity and in Section 4, we present an alternative proof based on the constraint subspace optimization method. We extend the convergence proof to the inexact local solver in Section 5. In Section 6, 7, and 8, we apply our method to mixed finite element methods for the Poisson and Darcy equations, non-conforming finite element methods for the Poisson equation, and mixed finite element methods for the Stokes equations, respectively. In the last section, we give conclusion and outlook for future work.

\section{Algorithm}\label{sec:algorithm}
Let $\mcal H$ be a Hilbert space equipped with inner product $(\cdot,\cdot)$ and $\mcal V\subset \mcal H$ be a closed subspace and thus $\mcal V$ is also a Hilbert space. Suppose $A: \mcal V\to \mcal V$ is a symmetric and positive definite (SPD) operator with respect to $(\cdot,\cdot)$, which introduces a new inner product $(u,v)_A := (Au,v)=(u,Av)$ on $\mcal V$. The norm associated to $(\cdot,\cdot)$ or $(\cdot,\cdot)_A$ will be denoted by $\|\cdot\|$ or $\|\cdot\|_A$, respectively. Let $\mcal P$ be another Hilbert space and let $B: \mcal V\to \mcal P$ be a linear  operator. With a slight abuse of notation, we still denote the inner product of $\mcal P$ by $(\cdot,\cdot)$. In most problems of consideration, the inner product $(\cdot,\cdot)$ for $\mcal H$ is the vector $L^2$-inner product while for $\mcal P$ it is the scalar $L^2$-inner product. The transpose $B^T: \mcal P \to \mcal V$ is the adjoint of $B$ in the $(\cdot, \cdot)$ inner product, i.e., $(Bv, q) = (v, B^Tq)$ for all $v\in \mcal V, q\in \mcal P$.


For an $f\in \mcal H$, we define the Dirichlet-type energy:
\begin{equation}
E(v) = \frac{1}{2}\|v\|_A^2 - ( f,v ), \quad \text{ for } v\in \mcal V.
\end{equation}
In this paper we always identify a functional in the dual space $\mcal H'$ as an element in $\mcal H$ through the Riesz map induced by $(\cdot,\cdot)$. Denote by $\mcal K = \ker(B)$ the subspace satisfying the constraint $Bv=0$, i.e., the null space of $B$. We are interested in the following constrained minimization problem:
\begin{equation}\label{main-opt}
\min _{v\in \mcal K} E(v).
\end{equation}

Since the energy is quadratic and convex, there exists a unique solution to~\eqref{main-opt} and the minimizer $u$ of~\eqref{main-opt} is characterized as the solution of the equation: Find $u\in \mcal K$ such that
\begin{equation}\label{spd}
(Au, v) = ( f, v ) \quad \text{ for all } v\in \mcal K.
\end{equation}
We introduce the operator $A_{\mcal K}: \mcal K\to \mcal K$ as $(A_{\mcal K} u, v) = (Au, v)$ for all $u, v\in \mcal K$ and the operator $Q_{\mcal K}: \mcal H\to \mcal K$ as the $(\cdot,\cdot)$-projection, i.e., for a given function $f\in \mcal H$, $Q_{\mcal K}f\in \mcal K$ satisfies $(Q_{\mcal K} f, v) = (f, v)$ for all $v\in \mcal K$. Then the operator form of~\eqref{spd} is: Find $u\in \mcal K$ such that
\begin{equation}\label{AK}
A_{\mcal K} u = Q_{\mcal K}f \quad \text{in } \mcal K.
\end{equation}
As it might be difficult to find bases for the subspace $\mcal K$, instead of solving the symmetric positive definite formulation~\eqref{AK}, we shall consider an equivalent saddle point formulation.

Let us introduce the Lagrange multiplier $p\in \mcal P$, equation~\eqref{spd} can be rewritten as the following saddle point system: Find $u\in \mcal V, p\in \mcal P$ such that
\begin{align*}
(Au, v) + (p, Bv) &= ( f, v ) &\text{for all } v\in \mcal V,\\
(Bu, q) \qquad \quad \quad \; & = 0&\text{for all } q\in \mcal P,
\end{align*}
which will be written in the operator form
\begin{equation}\label{ABBO}
\begin{pmatrix}
A & B^T\\
B & O
\end{pmatrix}
\begin{pmatrix}
u\\
p
\end{pmatrix}
=
\begin{pmatrix}
f\\
0
\end{pmatrix}.
\end{equation}

Let $\|\cdot\|_V$ and $\|\cdot\|_P$ be two appropriate norms for space $\mcal V$ and $\mcal P$, respectively. It is well known that~\eqref{ABBO} is well posed if and only if the following so-called Brezzi conditions \cite{Brezzi.F;Fortin.M1991} hold:
\begin{enumerate}
\item Continuity of operators $A$ and $B$: there exist constants $c_a, c_b > 0$ such that
$$
(Au, v) \leq c_a\|u\|_{V}\|v\|_{V}, \quad (Bv, q)\leq c_b\|v\|_{V}\|q\|_P, \quad \text{for all } u, v \in \mcal V, q\in \mcal P.
$$
\item Coercivity of $A$ in the kernel space. There exists a constant $\alpha>0$ such that
$$
(Au, u) \geq \alpha \|u\|_{V}^2 \quad \text{for all }u \in \ker(B).
$$

\item Inf-sup condition of $B$. There exists a constant $\beta >0$ such that
$$
\inf_{p\in \mcal P, p\neq 0} \sup_{v \in \mcal V, \tau \neq 0} \frac{(Bv, p)}{\|v\|_{V}\|p\|_{P}} \geq \beta.
$$
\end{enumerate}
%
%
%
%
Choices of norms $\| \cdot \|_V$ and $\| \cdot \|_P$ are not unique~\cite{Zulehner2011} and $\| \cdot \|_V = \| \cdot \|_A$ may not be always a good choice since $B$ may not be continuous in $\|\cdot \|_A$ norm, c.f. the mixed formulation of Poisson equation in Section \ref{sec:mixPoisson}. Throughout this paper, we will assume the well-posedness of~\eqref{ABBO} and focus on its efficient solvers. 

Problems~\eqref{AK} and~\eqref{ABBO} are equivalent theoretically but will lead to different algorithms. In practice, the saddle point formulation will be easier to solve when bases of $\mcal K$ are not available or expensive to form.



We shall develop and analyze multigrid methods for solving the saddle point system \eqref{ABBO} based on subspace correction methods~\cite{Xu1992} and its adaptation to optimization problems~\cite{Tai2001,Tai.X2003}. Let
$$
\mcal V = \mcal V_1 + \mcal V_2 + \cdots + \mcal V_N, \; \mcal V_i\subset \mcal V, i=1,\ldots, N,
$$
be a space decomposition of $\mcal V$ satisfying the condition 
$$
\mcal K = \mcal K_1 + \mcal K_2 + \cdots + \mcal K_N, \; \mcal K_i = \mcal V_i \cap \ker(B), i=1,\ldots, N.
$$
For $k\geq 0$ and a given approximated solution $u^k\in \mcal K$, one step of the Successive Subspace Optimization (SSO) method~\cite{Tai2001} is as follows:
\begin{figure}[htbp]
\begin{center}
\includegraphics[width = 4.2in]{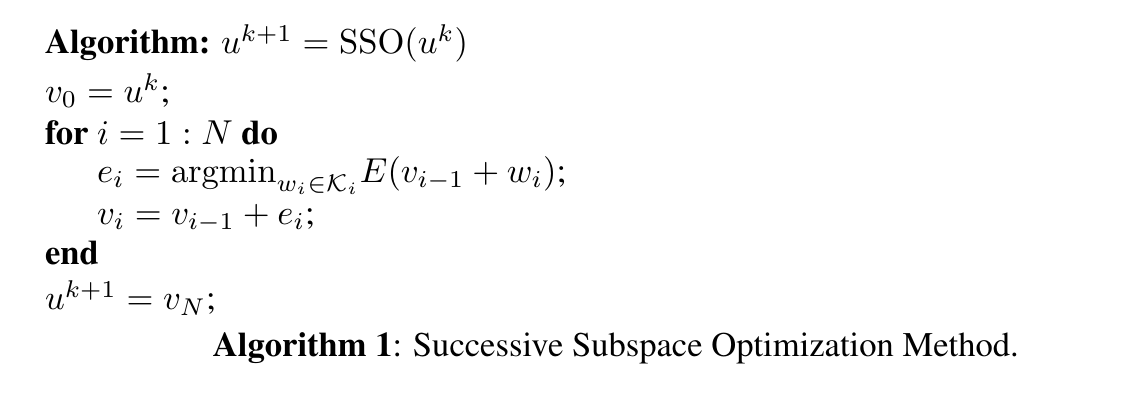}
\end{center}
\end{figure}


If we write the Euler equation of the local minimization problem, it reads as
\begin{equation}\label{correction}
(A e_i, \phi_i) = ( f - Av_{i-1}, \phi_i )\quad \text{for all } \phi_i\in \mcal K_i.
\end{equation}
Namely $e_i$ is the solution of the residual equation restrict to $\mcal K_i$. We can thus treat SSO as the subspace correction method for solving \eqref{spd} using the space decomposition $\mcal K = \sum _{i=1}^N\mcal K_i$. We can analyze the convergence from this point of view.

Using the fact $Au = f$ in $\mcal K'$ and $v_i = v_{i-1} + e_i$, equation~\eqref{correction} is also equivalent to the $A$-orthogonality
\begin{equation}\label{orth}
(u - v_i, \phi_i)_A = 0 \quad \text{for all } \phi_i\in \mcal K_i,
\end{equation}
which can be also written as 
\begin{equation}\label{orth2}
(E'(v_i), \phi_i) = 0 \quad \text{for all } \phi_i\in \mcal K_i.
\end{equation}

Let $\mcal P_i = \mcal P\cap B(\mcal V_i)$. Define $A_i: \mcal V_i \to \mcal V_i$ as for $u_i\in \mcal V_i$, $Au_i\in \mcal V_i$ such that $(A_i u_i, v_i) = (Au_i, v_i)$ for all $v_i \in \mcal V_i$, and $B_i: \mcal V_i \to \mcal P_i$ 
as for $u_i\in \mcal V_i$, $Bu_i\in \mcal P_i$ such that $(B_i u_i, q_i) = (Bu_i, q_i)$ for all $q_i \in \mcal P_i$. Let $Q_i: \mcal H \to \mcal V_i$ be the projection in $(\cdot,\cdot)$ inner product. 
The constrained minimization problem in the constraint subspace $\mcal K_i$ will be solved by solving a small saddle point system in $\mcal V_i$:
\begin{equation}\label{localproblem}
\begin{pmatrix}
A_i & B^T_i\\
B_i & O
\end{pmatrix}
\begin{pmatrix}
e_i\\
p_i
\end{pmatrix}
=
\begin{pmatrix}
Q_i(f - Av_{i-1})\\
0
\end{pmatrix}.
\end{equation}


A typical multilevel decomposition is given as follows. First we construct a macro-decomposition $\mcal V = \sum_{k=1}^J \mcal V_k$ with nested subspaces $\mcal V_1\subset \mcal V_2 \subset \ldots \subset \mcal V_J = \mcal V$. Usually they are based on a sequence of successively refined meshes. For each subspace $\mcal V_k, k=1, \ldots, J$, we introduce a micro-decomposition $\mcal V_k = \sum_{i=1}^{N_k}\mcal V_{k,i}$ and set $\mcal K_{k,i} = \mcal V_{k,i}\cap \ker(B)$. Note that the assumption $\mcal K = \sum_{k=1}^{J}\sum_{i=1}^{N_k}\mcal K_{k,i}$ requires a careful choice of the micro-decomposition of $\mcal V_k$. Roughly speaking, each subspace $\mcal V_{k,i}$ should be big enough to contain a basis function of $\mcal K$ and each basis function of $\mcal K$ should be contained in at least one $\mcal V_{k,i}$. Similar decomposition is required to design robust multigrid methods for nearly singular system~\cite{Lee.Y;Wu.J;Xu.J;Zikatanov.L2006}. 

\begin{remark}\rm
Solving local saddle problems in $\mcal V_{k,i}$ sequentially in the $k$-th level can be interpret as a multiplicative Schwarz smoother which is better known as the Vanka smoother~\cite{Vanka1986} for Navier-Stokes equations. \qed

%
%
%

\end{remark}

Due to the nestedness of the macro-decomposition, restriction and prolongation operators are needed only for two consecutive levels. In summary, SSO based on this multilevel decomposition leads to a V-cycle multigrid method for the saddle point problem~\eqref{ABBO} with a multiplicative Schwarz smoother. 

Thanks to the assumption $\mcal K_{k,i}\subset \mcal K$, if $u^k\in \mcal K$, then $u^{k+1} =$ SSO$(u^k)$ is still in $\mcal K$. Namely the iteration remains in the constrained subspace.
Uzawa method~\cite{Uzawa1958}, another popular iterative method for solving the saddle point problem, will not preserve the constraint and thus is not considered here.

We shall use either the unconstrained SPD formulation~\eqref{spd} and~\eqref{correction} or constrained saddle point formulation~\eqref{ABBO} and~\eqref{localproblem}. They are equivalent forms for the convergence analysis but different algorithmically.

We end this section with a discussion of the non-homogenous constraint, i.e., the saddle point problem
\begin{equation}\label{ABBOfg}
\begin{pmatrix}
A & B^T\\
B & O
\end{pmatrix}
\begin{pmatrix}
u\\
p
\end{pmatrix}
=
\begin{pmatrix}
f\\
g
\end{pmatrix}.
\end{equation}
To change to the form \eqref{ABBO}, we can first find a $u_*\in \mcal V$ satisfying $Bu_* = g$ and let $u = u_* + \delta u$. Then the equation for $\delta u$ is in the form \eqref{ABBO}.

There are several ways to find such $u_*$. One way is to solve 
\begin{equation}\label{ABBO0g}
\begin{pmatrix}
I & B^T\\
B & O
\end{pmatrix}
\begin{pmatrix}
u_*\\
p_*
\end{pmatrix}
=
\begin{pmatrix}
0\\
g
\end{pmatrix},
\end{equation}
which is supposed to be easier than solving \eqref{ABBOfg}. For Stokes equations, solving \eqref{ABBO0g} essentially requires a Poisson solver for pressure for which fast solvers are available. For Darcy equations, $A$ is a weighted mass matrix with possibly highly oscillatory coefficients, while \eqref{ABBO0g} is again just a Poisson operator. 

When the space $\mcal P$ consists of discontinuous elements, which is the case of most applications considered in this paper, we can find such $u_*$ by one V-cycle with post-smoothing only; see Section \ref{sec:mixPoisson} for details.


\section{Convergence Analysis based on the XZ identity}
In this section, we provide a convergence analysis using the SPD formulation~\eqref{spd} and~\eqref{correction}. The analysis is based on the XZ identity~\cite{Xu.J;Zikatanov.L2002} for the multiplicative iterative methods and can be found in~\cite{Xu.J;Chen.L;Nochetto.R2009}.

Denoted by $P_i$ the $A$-orthogonal projection onto $\mcal K_i$ for $i=1, \ldots, N$. Then the error operator of SSO can be written as $(I - P_N)(I-P_{N-1})\cdots (I-P_1)$, i.e., $u - u^{k+1} = \prod _{i=1}^N (I-P_{i}) (u-u^k)$, where $u^{k+1} = SSO(u^k)$. The following XZ identity was established in~\cite{Xu.J;Zikatanov.L2002}
\begin{equation}\label{XZ}
\Big \|\prod _{i=1}^N (I-P_{i}) \Big \|_A^2 = 1 - \frac{1}{1+c_0},
\end{equation}
where 
$$
c_0 = \sup _{\|v\|_A=1}\inf _{\sum
_{i=1}^Jv_{i}=v, \\
v_i\in \mcal K_i} \sum _{i=1}^N \Big \|P_i\sum _{j=i+1}^Jv_j \Big \|^2_{A}.
$$
For an elementary proof of \eqref{XZ}, we refer to Chen~\cite{Chen.L2009c}.


In order to estimate the constant $c_0$, we propose two important properties of the space decomposition.

\smallskip

\noindent{\bf Stable decomposition (SD)}: for every $v\in \mcal K$, there exists $v_i\in \mcal K_i, i=1, \ldots, N$ such that
$$
v = \sum _{i=1}^N v_i, \quad \text{ and }\quad \sum_{i=1}^N \|v_i\|_A^2 \leq C_A\|v\|_A^2.
$$

\noindent{\bf Strengthened Cauchy Schwarz inequality (SCS)}: for any $u_i\in \mcal K_i$ and $v_j\in \mcal K_j$
$$
\sum_{i=1}^N\sum_{j=i+1}^N (u_i, v_j)_A \leq C_S^{1/2} \left (\sum_{i=1}^N \|u_i\|^2_A\right )^{1/2}\left (\sum_{j=1}^N \|v_j\|^2_A\right )^{1/2}.
$$
With assumptions (SD) and (SCS), we shall provide a upper bound of $c_0$ and thus obtain a convergence proof of SSO method for solving the saddle point problem~\eqref{ABBO}. 
\begin{theorem}\label{th:exactrate}
Assume that  the space decomposition $\mcal K = \sum_{i=1}^N \mcal K_i$ satisfy assumptions (SD) and (SCS). For SSO method, we have
$$
\Big \|\prod _{i=1}^N (I-P_{i})\Big \|^2_A \leq 1-\frac{1}{1+C_AC_S}. 
$$
\end{theorem}
\begin{proof}
We apply (SCS) with $u_i = P_i \sum_{j=i+1}^N v_j$ to obtain
  \begin{align*}
\sum _{i=1}^N \|u_i\|_A^2 &= \sum _{i=1}^N (u_i,P_i \sum
  _{j=i+1}^N v_j)_A = \sum _{i=1}^N \sum _{j=i+1}^N (u_i, v_j)_A\\  
&\leq C_S^{1/2}\left (\sum _{i=1}^N \|u_i\|_A^2\right)^{1/2}
\left (\sum _{i=1}^N \|v_i\|_{A}^2 \right)^{1/2},
  \end{align*}
which leads to the inequality 
\begin{equation}
\sum _{i=1}^N \|u_i\|_A^2 \leq C_S\sum _{i=1}^N \|v_i\|_{A}^2.
\end{equation}

Consequently, we choose $v=\sum _{i=1}^N v_i$ as a stable decomposition satisfying (SD) to get
$$
\sum _{i=1}^N \Big \|P_{i} \sum _{j=i+1}^N v_{j}\Big \|_{A}^2=\sum
_{i=1}^N \|u_i\|_A^2\leq C_S\sum _{i=1}^N \|v_i\|_{A}^2 \leq
C_SC_A\|v\|_A^2,
$$
which implies $c_0\leq C_SC_A$. The desired result then follows from the X-Z identity~\eqref{XZ}.
\end{proof}


The assumption (SCS) is relatively easy to verify. The key is to construct a stable decomposition of the constraint space $\mcal K$. 

\section{Convergence Analysis based on Constrained Optimization}
In this section we provide an alternative proof using the constraint optimization approach established by Tai~\cite{Tai.X2003}. It also provides a better approach to extend the convergence proof to inexact and/or nonlinear local solvers. 

We will always denote by $u$ the global minimizer of \eqref{main-opt}. Given an initial guess $u^0\in \mcal K$,  let $u^k$ be the $k$th iteration in SSO algorithm for $k=1,2,\cdots$. We aim to prove a linear reduction of the energy difference
\begin{equation}\label{linearreduction}
E(u^{k+1}) - E(u) \leq \rho \left [ E(u^{k}) - E(u) \right ],
\end{equation}
with a contraction factor $\rho \in (0,1)$. Ideally $\rho$ is independent of the size of the problem. The proof is developed in~\cite{Tai2001,Tai.X2003} for a nonlinear and convex energy but simplified here for the quadratic energy.

We first explore the relation between the energy and the $A$-norm of the error. 
\begin{lemma}\label{lm:energyandnorm}
For any $w, v\in \mcal V$, we have
\begin{equation}\label{Ewv}
E(w)-E(v) = \frac{1}{2}\|w - v\|_A^2 + ( E'(v), w-v ).
\end{equation}
Consequently for the minimizer $u\in \mcal K$ and any $w\in \mcal K$, 
\begin{equation}\label{Ewu}
E(w)-E(u) = \frac{1}{2}\|w - u\|_A^2.
\end{equation}
\end{lemma}
\begin{proof}
Verification of~\eqref{Ewv} and~\eqref{Ewu} is straightforward.
\end{proof}
Based on the identity~\eqref{Ewu}, the target inequality~\eqref{linearreduction} becomes a more familiar one 
\begin{equation}
\|u^{k+1} - u\|_A \leq \rho^{1/2}\|u^k - u\|_A. 
\end{equation}

Let $d_k = E(u^k) - E(u)$ and $\delta _k = E(u^k) - E(u^{k+1}).$ The quantity $d_k$ is the distance of the current energy to the lowest one, $\delta_k$ is the amount of the energy decreased in one iteration, and they are connected by the identity $\delta_k = d_k - d_{k+1}$. By Lemma \ref{lm:energyandnorm}, we have $d_k = \frac{1}{2}\|u^k - u\|_A^2$ but in general $\delta_k \neq \frac{1}{2}\|u^k - u^{k+1}\|_A^2$ since $u^{k+1}$ may not be the minimizer. For each $v_i, i=1,\ldots,N,$ in SSO, we do have  $$E(v_{i-1}) - E(v_i) = \frac{1}{2}\|v_{i-1} - v_i\|_A^2,$$ since $v_i$ is the local minimizer and $v_{i-1} - v_i = - e_i\in \mcal K_i$; see also the orthogonality~\eqref{orth2}. 
Borrowing the terminology of the convergence theory of adaptive finite element methods~\cite{Nochetto.R;Siebert.K;Veeser.A2009}, we shall present our proof based on the following two inequalities. 

\medskip
\noindent{\bf Discrete Lower Bound.} There exists a positive constant $C_L$ such that for $k=0,1,2, \ldots$
$$
\delta _k \geq C_L \sum _{i=1}^N \|e_i\|^2_A.
$$

\noindent{\bf Upper Bound.} There exists a positive constant $C_U$ such that for $k=0,1,2, \ldots$
$$
d_{k+1} \leq C_U \sum _{i=1}^N \|e_i\|^2_A.
$$

\begin{theorem}
Assume that  the discrete lower bound and upper bound hold with constants $C_L$ and $C_U$ respectively. We then have
$$
d_{k+1}\leq \frac{c_0}{1+c_0}d_k,
$$
where $c_0=C_U/C_L$.
\end{theorem}
\begin{proof}
The proof is straightforward by assumptions and rearrangement of the following inequality
$$
d_{k+1} \leq C_U \sum _{i=1}^N \|e_i\|^2_A \leq C_U/C_L  \delta _k = c_0( d_k  - d_{k+1}).
$$
\end{proof}

Verifying the lower bound is relatively easy since $E$ is convex. Indeed we have the following identity which characterizes exactly the amount of energy decreased in one step of SSO. Again in the sequel, $u^{k+1}={\rm SSO}(u^k)$ and $e_i$ is the $i$th correction in $\mcal K_i$, for $i=1,\ldots,N$.
\begin{theorem}
$$
 E(u^k) - E(u^{k+1}) = \frac{1}{2} \sum _{i=1}^N \|e_i\|^2_A.
$$
\end{theorem}
\begin{proof}
By the identity \eqref{Ewv} and the orthogonality \eqref{orth2}, we have, for $i=1,\ldots, N$,
$$
E(v_{i-1}) - E(v_i) = \frac{1}{2}\|v_{i-1}-v_i\|^2_A = \frac{1}{2}\|e_i\|^2_A,
$$
and consequently
$$
 E(u^k) - E(u^{k+1}) = \sum _{i=1}^N\left [  E(v_{i-1}) - E(v_i)\right ] =  \frac{1}{2} \sum _{i=1}^N \|e_i\|^2_A.
$$
\end{proof}

Proving the upper bound is more delicate. We first present a lemma which can be verified directly by definition and Lemma \ref{lm:energyandnorm}.
\begin{lemma}\label{lm:E'u}
\begin{equation}\label{eq:E'u}
(E'(u^{k+1}) - E'(u), u^{k+1} - u) = \|u^{k+1}-u\|_A^2 = 2\left [E(u^{k+1}) - E(u)\right ].
\end{equation}
\end{lemma}

We then give a multilevel decomposition of the left-hand side of \eqref{eq:E'u}.
\begin{lemma}\label{lm:doublesum}
For any decomposition $u^{k+1}-u = \sum _{i=1}^N w_i, w_i\in \mcal K_i, i=1, 2, \ldots, N$,
$$
(E'(u^{k+1}) - E'(u), u^{k+1} - u) = \sum _{i=1}^N\sum _{j>i}^N (e_j, w_i )_A.
$$
\end{lemma}
\begin{proof}
\begin{align*}
 &(  E'(u^{k+1}) - E'(u), u^{k+1} - u )  = (  E'(u^{k+1}) , u^{k+1} - u )  \\
&= \sum _{i=1}^N(  E'(u^{k+1}) - E'(v_i), w_i )  = \sum _{i=1}^N\sum _{j>i}^N (  E'(v_j) - E'(v_{j-1}), w_i ) = \sum _{i=1}^N\sum _{j>i}^N ( e_j, w_i )_A.
\end{align*}
In the first step, we use the fact $E'(u) = 0$ in $\mcal K'$ since $u$ is the minimizer and $u^{k+1}-u \in \mcal K$. In the second step we use $E'(v_i) = 0$ in $\mcal K_i'$ since $v_i$ is the minimizer in $\mcal K_i$ and $w_i\in \mcal K_i$; see also~\eqref{orth}. 
\end{proof}

\begin{lemma}\label{lm:upper}
Assume that  the space decomposition satisfies assumptions (SD) and (SCS). Then we have the upper bound
$$
E(u^{k+1}) - E(u) \leq \frac{1}{2}C_SC_A\sum _{i=1}^N\|e_i\|^2_A.
$$
\end{lemma}
\begin{proof}
We shall chose a stable decomposition for $u^{k+1} - u = \sum _{i=1}^N w_i, w_i\in \mcal K_i, i=1, 2, \ldots, N$. By Lemma \ref{lm:doublesum} and (SCS), we have
\begin{align*}
(  E'(u^{k+1}) - E'(u), u^{k+1} - u )  &= \sum _{i=1}^N\sum _{j>i}^N (  e_j, w_i )_A \\
&\leq C_S^{1/2}\left (\sum_{j=1}^N \|e_j\|^2_A\right )^{1/2}\left (\sum_{i=1}^N \|w_i\|^2_A\right )^{1/2}\\
&\leq (C_SC_A)^{1/2}\left (\sum_{i=1}^N \|e_i\|^2_A\right )^{1/2}\|u^{k+1} - u\|_A.
\end{align*}
Substituting the identity (see Lemma \ref{lm:E'u})
$$
\|u^{k+1} - u\|_A^2 = (  E'(u^{k+1}) - E'(u), u^{k+1} - u )
$$
into the above inequality and canceling one $\|u^{k+1} - u\|_A$, we can obtain 
$$
\|u^{k+1} - u\|_A^2 \leq C_SC_A\sum_{i=1}^N\|e_i\|_A^2.
$$
Using the identity $E(u^{k+1}) - E(u) = \|u^{k+1} - u\|_A^2/2$, we  obtain the desired result.
\end{proof}

We summarize our convergence result into the following theorem.
\begin{theorem}
Assume that  the space decomposition $\mcal K = \sum_{i=1}^N \mcal K_i$ satisfies assumptions (SD) and (SCS). Then
$$
E(u^{k+1}) - E(u)\leq \left (1-\frac{1}{1+C_AC_S}\right) \left [E(u^k) - E(u)\right ]. 
$$
\end{theorem}
\begin{remark}\rm
The estimate is consistent with the one obtained by the XZ identity which indicates that our energy estimate is sharp. 
\end{remark}

\section{Convergence Analysis with Inexact Local Solvers}
In the algorithm SSO, we assume that the local problem is solved exactly which may be costly when the dimension of the local space is large. In this section, we consider inexact solvers using one gradient iteration and establish the corresponding convergence proof. Note that XZ identity cannot be applied to the nonlinear solvers considered here.

Recall that the local constrained minimization problem is: let $r_i = Q_i(f - Av_{i-1})$, find $e_i^*\in \mcal K_i$  such that
\begin{equation}\label{localstokes}
\begin{pmatrix}
A_i & B_i^T\\
B_i & O
\end{pmatrix}
\begin{pmatrix}
e_i^*\\
p_i^*
\end{pmatrix}
=
\begin{pmatrix}
r_i\\
0
\end{pmatrix}.
\end{equation}
Here we use $e_i^*, p_i^*$ to denote the solution obtained by the exact solver. In the inexact solver proposed below, the constraint is still satisfied but operator $A_i$ is replaced by a simpler one $D_i$, e.g., the diagonal of $A_i$. In general, let $D_i$ be an SPD operator on $\mcal V_i$, we first solve the local problem
\begin{equation}\label{localapproximatedstokes}
\begin{pmatrix}
D_i & B_i^T\\
B_i & O
\end{pmatrix}
\begin{pmatrix}
s_i\\
p_i
\end{pmatrix}
=
\begin{pmatrix}
r_i\\
0
\end{pmatrix}.
\end{equation}
Then we apply the line search along the direction $s_i$ to find an optimal scaling:
\begin{equation}\label{linesearch}
\min _{\alpha \in \mbb R}E(v_{i-1} - \alpha s_i),
\end{equation}
whose solution is 
\begin{equation}\label{alpha}
\alpha = \frac{(r_i,s_i)}{(As_i, s_i)}.
\end{equation}
We update $$v_i = v_{i-1} - \alpha s_i.$$ This is one step of a preconditioned gradient method and $D_i$ is a preconditioner of $A_i$.

In this section, we will always denote by $e_i^*$ the solution of~\eqref{localstokes} and $e_i = \alpha s_i$ with $s_i$ being the solution of~\eqref{localapproximatedstokes} and $\alpha$ giving by~\eqref{alpha}. 
With such choice of $\alpha$, we still have the first order condition 
\begin{equation}\label{localorth}
(E'(v_i), e_i) = 0.
\end{equation}

\begin{remark}\rm
In the original Vanka smoother for Navier-Stokes equation, $D_i = \omega \,{\rm diag} (A_i)$ with a suitable parameter $\omega \in (0.5,0.8)$~\cite{Vanka1986} and no line search is applied, i.e., $\alpha = 1$.
$\Box$\end{remark}
 
Using the first order condition~\eqref{localorth}, we still have the following identity. 
\begin{lemma}
$$
E(u^k) - E(u^{k+1}) = \frac{1}{2} \sum _{i=1}^N \|e_i\|^2_A.
$$
\end{lemma}

Again the upper bound is more delicate. 
We first adapt the analysis in~\cite{Braess1999} to establish the following inequalities. Recall that for an SPD operator $M$, $\kappa (M) = \lambda_{\max}(M)/\lambda_{\min}(M)$ is the condition number of $M$.
\begin{lemma}For the inexact local solver described above, we have
\begin{equation}\label{localepsilon}
\|e_i^* - e_i\|_A \leq \epsilon \|e_i^*\|_A, \quad\text{ with } \epsilon = \frac{\kappa (D^{-1}_iA_i) - 1}{\kappa (D^{-1}_iA_i) + 1} \in (0,1).
\end{equation}
Consequently by the triangle inequality
\begin{equation}\label{localepsilon2}
\|e_i^* - e_i\|_A \leq \frac{\epsilon}{1-\epsilon} \|e_i\|_A = \frac{1}{2}(\kappa(D_i^{-1}A_i)-1)\|e_i\|_A.
\end{equation}
\end{lemma}
\begin{proof}
To simplify the notation, we suppress the subscript $i$ in the proof. Let $\tilde e = \omega s$ where $s$ is determined by \eqref{localapproximatedstokes} and $\omega \in \mbb R$ is a parameter. Then following~\cite{Braess.D;Sarazin.R1997,Braess1999} we have the error equation
\begin{equation}
e^* - \tilde e = P_D(I-\omega D^{-1}A)e^* = (I - \omega D^{-1}A_\mcal K)e^*,
\end{equation}
where $P_D = I - D^{-1}B^T(BD^{-1}B^T)^{-1}B$ is the projection to $\mcal K$ in the $(\cdot, \cdot)_D:=(D\cdot,\cdot)$ inner product, and $A_\mcal K = DP_DD^{-1}AP_D$. 

Note that $P_D$ is symmetric in $(\cdot, \cdot)_D$. We can then verify $A_\mcal K$ is symmetric and semi-positive definite and $(\cdot, \cdot)_{A_\mcal K} = (\cdot, \cdot)_A$ restricted to $\mcal K$. Since the operator $D^{-1}A_\mcal K$ is symmetric w.r.t. $(\cdot, \cdot)_{A_\mcal K}$ and $e^*, \tilde e\in \mcal K$, we have
$$
\|e^* - \tilde e\|_A = \|e^* - \tilde e\|_{A_\mcal K}\leq \| I - \omega D^{-1}A_{\mcal K} \|_{A_\mcal K} \|e^*\|_{A}.
$$

By subtracting a fixed energy $E(v_i^*)$ from $E(v_{i-1}-\alpha s_i)$, it is easy to see the line search~\eqref{linesearch} is equivalent to $\min _{\alpha\in \mbb R}\|e_i^* - \alpha s_i\|_A$. Therefore
$$
\|e^* - e\|_A\leq \|e^* - \tilde e\|_A\leq \| I - \omega D^{-1}A_{\mcal K} \|_{A_\mcal K} \|e^*\|_{A}.
$$
Consequently
$$
\|e^* - e\|_A\leq \inf_{\omega \in \mbb R}\| I - \omega D^{-1}A_{\mcal K} \|_{A_{\mcal K}} \|e^*\|_{A} = \frac{\kappa (D^{-1}A_{\mcal K}) - 1}{\kappa (D^{-1}A_{\mcal K}) + 1}\|e^*\|_{A}.
$$
The condition number $\kappa (D^{-1}A_\mcal K)$, which is not easy to estimate since $A_{\mcal K}$ is not formed explicitly, can be bounded by
$$
\kappa (D^{-1}A_\mcal K) = \kappa (P_DD^{-1}AP_D)\leq \kappa(D^{-1}A).
$$
\end{proof}

\begin{lemma}
Assume that  the space decomposition satisfies assumptions (SD) and (SCS). 
For SSO with the local in-exact solver described in this section, we have
$$
E(u^{k+1}) - E(u) \leq \frac{1}{2}C_A\left [C_S^{1/2} + \frac{1}{2}\left (\max_{1\leq i\leq N}\kappa(D_i^{-1}A_i)-1\right ) \right ]^2 \sum _{i=1}^N\|e_i\|^2_A.
$$
\end{lemma}
\begin{proof}
As before, we chose a stable decomposition for $u^{k+1} - u = \sum _{i=1}^N w_i, w_i\in \mcal K_i, i=1, 2, \ldots, N$ and split as
\begin{align*}
(  E'(u^{k+1}) - E'(u), u^{k+1} - u )  & = (  E'(u^{k+1}) , u^{k+1} - u )  \\
&= \sum _{i=1}^N(  E'(u^{k+1}) - E'(v_i), w_i ) + (E'(v_i) - E'(v_i^*), w_i) \\
& = \sum _{i=1}^N\left [\sum _{j>i}^N (  E'(v_j) - E'(v_{j-1}), w_i ) + (v_i - v_i^*, w_i)_A\right ].
\end{align*}
The first term can be bounded as before, c.f. Lemma~\ref{lm:upper}
$$
 \sum _{i=1}^N\sum _{j>i}^N (  E'(v_j) - E'(v_{j-1}), w_i )
\leq (C_SC_A)^{1/2}\left (\sum_{i=1}^N \|e_i\|^2_A\right )^{1/2}\|u^{k+1} - u\|_A.
$$
For the second term, using \eqref{localepsilon2}, we have
\begin{align*}
\sum_{i=1}^N (v_i - v_i^*, w_i)_A &\leq \frac{\epsilon}{1-\epsilon}\sum_{i=1}^N\|e_i\|_A\|w_i\|_A \\
&\leq  \frac{\epsilon}{1-\epsilon} \left (\sum_{i=1}^N\|e_i\|_A^2\right )^{1/2}\left (\sum_{i=1}^N\|w_i\|_A^2\right )^{1/2}\\
&\leq  \frac{\epsilon}{1-\epsilon}C_A^{1/2} \left (\sum_{i=1}^N\|e_i\|_A^2\right )^{1/2}\|u-u^{k+1}\|_A.
\end{align*}
Combining these two estimates, we then get the desired result.
\end{proof}

\begin{theorem}
Assume that  the space decomposition satisfies assumptions (SD) and (SCS).  For SSO with the local in-exact solver described in this section, we have
$$
E(u^{k+1}) - E(u) \leq \rho \left [ E(u^{k}) - E(u) \right ],
$$
with contraction rate
$$
\rho = 1-\frac{1}{1+C_A\left [C_S^{1/2} + (\max_{1\leq i\leq N}\kappa(D_i^{-1}A_i)-1)/2\right ]^2}.
$$
\end{theorem}

We end this section with several remarks.
\begin{remark}\rm
To be an efficient local solver, $D_i$ is usually a diagonal matrix which may not be a good preconditioner for elliptic operators. The rate will deteriorate as $\epsilon$ becomes close to one, i.e., $\kappa(D^{-1}_iA_i)\gg 1$. On the other hand, for elliptic operators in $\mbb R^n$, and for $D_i = {\rm diag}(A_i)$, we have estimate $\kappa(D^{-1}_iA_i) \lesssim \dim (\mcal V_i)^{2/n}$~\cite{Bank.R;Scott.L1989}.
We can thus apply the estimate to a decomposition such that each local problem is of size $\mcal O(1)$.
$\Box$\end{remark}

\begin{remark}\rm
The solver considered here is one step of the preconditioned gradient method. The same analysis is applicable to a more efficient Preconditioned Conjugate-Gradient (PCG) solver with more than one iteration. The first order condition~\eqref{localorth} still holds for the PCG iterations. 
$\Box$\end{remark}

\begin{remark}\rm
The local gradient method is a nonlinear iterative method since the parameter $\alpha$ depends on the iteration. To prove the energy contraction for the linear constraint smoother, i.e., with a fixed parameter $\alpha$, we need to estimate the spectrum of the operator $P_DD^{-1}AP_D$ which is not easy since the projection $P_D$ is in a $L^2$-type inner product not the $A$-inner product. Technically, the first order condition~\eqref{localorth} may not hold for a fixed parameter $\alpha$.
$\Box$\end{remark}

\section{Application to Mixed methods for Poisson and Darcy Equation}\label{sec:mixPoisson}
In this section, we consider mixed finite element methods for solving the Poisson  equation and Darcy equation in two and three dimensions. Let $\Omega$ be a polygon or polyhedron domain and triangulated into a quasi-uniform mesh $\mcal T_h$ with mesh size $h$. Assume that $\mcal T_h$ is obtained by uniform refinements from an initial mesh $\mcal T_1$ of $\Omega$, i.e., there exists  a sequence of meshes $\mcal T_1, \mcal T_2, \ldots, \mcal T_J = \mcal T_h$. The triangulation $\mcal T_1$ is a shape regular triangulation of $\Omega$ and $\mcal T_{k+1}$ is obtained by dividing each element in $\mcal T_{k}$ into four congruent small elements (two dimensions) or eight small elements (three dimensions). The mesh size $\mcal T_k$ will be denoted by $h_k$. By the construction $h_k/h_{k+1} = 2$.

\subsection{Problem setting} We consider the Poisson equation with Neumann boundary
$$
-\Delta p = f \text{ in } \Omega, \quad \partial_n p = 0 \text{ on } \partial \Omega
$$ 
where $n$ is the outwards normal vector of $\partial \Omega$. Let $u = \nabla p$. We obtain the mixed formulation of Poisson equation: find $u\in H_0(\div;\Omega):=\{v\in (L^2(\Omega))^2, \div v \in L^2(\Omega), v\cdot n |_{\partial \Omega}= 0\}$, where $v\cdot n$ should be understood in the trace sense, and $p\in L^2_0(\Omega) := \{q\in L^2(\Omega), \int_{\Omega}q \dx = 0\}$ such that 
\begin{align*}
(u, v) - (\div v, p) &= 0, \quad \forall v\in H_0(\div;\Omega), \\
 -(\div u, q) &= (f, q), \quad \forall q\in L^2_0(\Omega).
\end{align*}

Choose finite element spaces $\mcal V \subset H_0(\div;\Omega)$ and $\mcal P \subset L^2_0(\Omega)$ so that the following sequence is exact
\begin{equation}\label{2Dexact}
\mcal S \stackrel{\curl}{\longrightarrow} \mcal V \stackrel{\div}{\longrightarrow} \mcal P \to 0,
\end{equation}
where $\mcal S$ is another appropriate finite element space. Choices of $\mcal S, \mcal V$, and $\mcal P$ will be made clear in the context. Subscript $k$ will be used when spaces are associated with triangulation $\mcal T_k$ and when $k=J$ the subscript will be suppressed.

The saddle point problem can be written as follows: Given $f\in \mcal P$, find $w\in \mcal V, p\in \mcal P$ such that
\begin{equation}\label{mixedpoisson0}
\begin{pmatrix}
M & B^T\\
B & O
\end{pmatrix}
\begin{pmatrix}
w\\
p
\end{pmatrix}
=
\begin{pmatrix}
0\\
f
\end{pmatrix},
\end{equation}
where $M$ is the mass matrix of $\mcal V $ and $B$ is the discretization of $-\div$ operator. For this problem, $A=M$ and the $A$-norm is just the standard $L^2$-norm. 

Our method and analysis can be readily adapted to the second order elliptic equation with variable coefficients $K$ i.e., Darcy equation, for which the constitutive equation becomes $(K^{-1} u, v) - (\div v, p) = 0$. The $A$-norm is a weighted $L^2$-norm and the exact sequence \eqref{2Dexact} still holds. The constant $C_A$, however, could depend on the condition number of $K$; see Remark \ref{rm:weightednorm}.

To apply our framework, we should first find a $u_*\in \mcal V$ satisfying $Bu_* = f$. 
Set $w= u_* + u$, the system \eqref{mixedpoisson0} can be changed to the form of \eqref{ABBO}:
\begin{equation}\label{mixedpoisson}
\begin{pmatrix}
M & B^T\\
B & O
\end{pmatrix}
\begin{pmatrix}
u\\
p
\end{pmatrix}
=
\begin{pmatrix}
-Mu_*\\
0
\end{pmatrix}.
\end{equation}
%
%
%
As discussed in Section \ref{sec:algorithm}, we can find such $u_*$ by solving $BB^Tu_* = Bf$. We now discuss a more efficient way utilizing the hierarchical structure of meshes. We start from a solution $u_*^1$ of \eqref{mixedpoisson0} on the coarsest mesh $\mcal T_1$ which can be found by direct solvers. For $k=1,\ldots, J-1$, when $u_*^k$ on $\mcal T_k$ with property $Bu_*^k = f$ holds element-wise on $\mcal T_k$ is found, for each element $T\in \mcal T_{k}$, we solve \eqref{mixedpoisson0} in $\mcal V_{k+1}$  restricted to $T$ and with boundary condition $u\cdot n|_{\partial T} = u_*^k\cdot n|_{\partial T}$. That is we use $\mcal T_k$ to get a domain decomposition of $\mcal T_{k+1}$ and $u_*^k$ as the boundary condition to decompose a global problem into local problems on elements. The local problem is well defined since the compatible condition is enforced by $Bu_*^k = f$ on $T$ and the solution of the local problem will give $u_*^{k+1}$ with the property $Bu_*^{k+1} = f$ for each element in $\mcal T_{k+1}$. The whole procedure is just one V-cycle with post-smoothing only and using a non-overlapping Schwarz method as a smoother. The computational cost is thus negligible. 

Thanks to the exact sequence \eqref{2Dexact}, we have a clear characterization of $\ker (B) = \curl (\mcal S)$ which will be helpful to construct a stable multilevel decomposition of $\mcal K$. Based on the hierarchy of the meshes, we have a macro-decomposition of $\mcal S = \sum _{k=1}^J \mcal S_k$. For each space $\mcal S_k$, we decompose into one dimensional subspaces $\Phi_{k,j}$ spanned by one basis function, i.e., $\mcal S_k = \sum_{j=1}^{N_k}\Phi_{k,j}$ with $N_k = \dim \mcal S_k$. Let $\Omega_{k,i}$ be the support of $\Phi_{k,i}$. We chose $\mcal V_{k,i} = H_0(\div;\Omega_{k,i})\cap \, \mcal V_k$ for $i = 1,\ldots, N_k$. 
Then we have the decomposition 
\begin{equation}\label{RTdec}
\mcal V = \sum_{k=1}^J\sum_{i=1}^{N_k}\mcal V_{k,i}, 
\end{equation}
and
$$
\mcal K = \sum_{k=1}^J\sum_{i=1}^{N_k}\mcal K_{k,i} 
\text{ with } 
\mcal K_{k,i} = \curl \Phi_{k,i}.
$$
%


We shall apply SSO based on the space decomposition $\mcal K = \sum_{k=1}^J\sum_{i=1}^{N_k}\mcal K_{k,i}$ and prove its uniform convergence. As an example, for the lowest order RT element \cite{Raviart.P;Thomas.J1977}, the smoother is an overlapping multiplicative Schwarz smoother requiring solving a small saddle point system at  in the patch of each vertex in two dimensions and in the patch of each edge in three dimensions.
Since the construction of a stable decomposition in two and three dimensions is different, we split the discussion into two subsections. 

\begin{remark}\rm
One can use a basis of $\mcal S$ to reduce the saddle point system into a SPD one and develop multigrid methods or domain decomposition methods for the SPD formulation; see e.g.~\cite{Ewing.R;Wang.J1992a,Hiptmair1999,Cai2003}. 
$\Box$\end{remark}

\subsection{Two dimensions}
In two dimensions, the space $\mcal S\subset H_0^1(\Omega)$ is a Lagrange element space based on the mesh $\mcal T_h$. 
To be specific, we will consider the important case when $\mcal S$ is the simplest linear finite element space, $\mcal V$ is the lowest order Raviart-Thomas element space~\cite{Raviart.P;Thomas.J1977}, and $\mcal P$ is the piecewise constant space.  Extension to high order elements is straightforward. The subspace $\mcal V_{k,i}$ is spanned by basis vectors of edges connecting to the $i$th vertex in triangulation $\mcal T_k$. 

We first verify the stable decomposition for the macro-decomposition $\mcal K = \sum_{k=1}^J\mcal K_{k}$. We denoted by $Q_k$ the $L^2$ projection $Q_k: \mcal S_J \to \mcal S_k$ for $k=1,\ldots, J$ and set $Q_0 = 0$. Note that due to the nestedness $Q_l Q_k = Q_l$ for $l\leq k$.
\begin{lemma}\label{lm:poissonmacro}
For every $v\in \mcal K$, there exists $v_k\in \mcal K_i, k=1,\ldots,J$ such that $v = \sum_{k=1}^J v_k$ and $\sum_{k=1}^J\|v_k\|^2\lesssim \|v\|^2$. 
\end{lemma}
\begin{proof}
In two dimensions, we have the relation $(\curl \phi, \curl \psi) = (\nabla \phi, \nabla \psi)$. Therefore the stable decomposition (SD) comes from that for the Lagrange elements. More specifically, since $u\in \mcal K$, there exists a unique $\phi \in \mcal S$ such that $u = \curl \phi$. We then chose the $H^1$-stable decomposition of $\phi$ as $\phi = \sum_{k=1}^J(Q_k - Q_{k-1})\phi$ and let $u_k = \curl (Q_k - Q_{k-1})\phi$. The stable decomposition for the decomposition $u = \sum_{k}u_k$ in $M$-norm is equivalent to that of $\phi = \sum_{k=1}^J(Q_k - Q_{k-1})\phi$ in $H^1$-norm which is well known; see e.g.~\cite{Xu1992}.
\end{proof}

\begin{remark}\label{rm:weightednorm}\rm
 For Darcy equation with variable coefficients $K$, the $A$-norm of $v$ is changed to a weighted $H^1$ norm of $\phi$ for $v = \curl \phi$. If assuming $K$ is piecewise constant on the coarsest mesh, we can find a multilevel decomposition using hierarchical basis such that the inequality $\sum_{k=1}^J\|v_k\|^2\leq C|\log h|\|v\|^2$ holds with a penalty factor $|\log h|$ but with a constant $C$ independent of the variation of $K$; see \cite{Bank1988}.
 \qed
\end{remark}
We then verify the micro-decomposition is stable. 
\begin{lemma}\label{lm:poissonmicro}
Let $\phi_k = (Q_k - Q_{k-1})\phi = \sum_{i=1}^{N_k}\phi_{k,i}$ be the nodal basis decomposition and let $u_{k,i} = \curl \phi_{k,i}$. Then the decomposition $u_k = \sum_{i=1}^{N_k} u_{k,i}$ is stable in $L^2$-norm.
\end{lemma}
\begin{proof}
We apply the inverse inequality and the stability of the nodal basis decomposition in $L^2$-norm to get
$$
\sum_{i=1}^{N_k} \|u_{k,i}\|^2 = \sum_{i=1}^{N_k} \| \curl \phi_{k,i}\|^2\lesssim h_k^{-2}\sum_{i=1}^{N_k} \| \phi_{k,i}\|^2\lesssim h_k^{-2}\| \phi_{k}\|^2.
$$
We write the term $\phi_k = (Q_k - Q_{k-1})\phi = (I- Q_{k-1}) (Q_k - Q_{k-1})\phi$ and bound it as
$$
\|\phi_k\| \lesssim h_k\| \curl (Q_k - Q_{k-1})\phi\| = h_k \|u_k\|.
$$
The desired inequality then follows.
\end{proof}

Combination of Lemma \ref{lm:poissonmacro} and \ref{lm:poissonmicro} leads to a stable multilevel decomposition. 
\begin{theorem}\label{th:SDpoisson}
For every $v\in \mcal K$, there exists $v_{k,i}\in \mcal K_{k,i}, k=1,\ldots,J, i=1,\ldots, N_k$ such that $v = \sum_{k=1}^J\sum_{i=1}^{N_k} v_{k,i}$ and $\sum_{k=1}^J\sum_{i=1}^{N_k}\|v_{k,i}\|^2\lesssim \|v\|^2$. 
\end{theorem}

To verify assumption (SCS), we first present the following inequality and refer to \cite{Xu1992} for a proof. 
\begin{lemma}\label{lm:SCS}
For any $\phi_k\in \mcal S_k, \phi_l\in \mcal S_l, l\geq k$, we have
$$
(\curl \phi_k, \curl \phi_l)
\lesssim 
\left (\frac{1}{2}\right )^{l-k}\|\curl \phi_k\| h_{l}^{-1}\|\phi_l\|.
$$
\end{lemma} 
We use the lexicographical order of the double index, i.e., $(l,j)>(k,i)$ if $l>k$ or $l=k,j>i$. 
\begin{theorem}For any $u_{k,i}\in \mcal K_{k,i}$ and $v_{l,j}\in \mcal K_{l,j}$, we have
$$
\sum_{k=1}^J\sum_{i=1}^{N_k} \sum_{(l,j)>(k,i)}(u_{k,i}, v_{l,j}) \lesssim \left (\sum_{k=1}^J\sum_{i=1}^{N_k}\|u_{k,i}\|^2 \right )^{1/2} \left (\sum_{l=1}^J\sum_{j=1}^{N_l}\|v_{l,j}\|^2\right )^{1/2}.
$$
\end{theorem}
\begin{proof}
We can write $u_{k,i} = \curl \phi_{k,i}$ and $v_{l,j} = \curl \psi_{l,j}$ for some $\phi_{k,i}\in \mcal S_k, \psi_{l,j}\in \mcal S_l$. We split the summation $\sum_{(l,j)>(k,i)}$ into two parts $\sum_{l>k}\sum_{j=1}^{N_l}$ and $\sum_{l=k,j>i}^{N_k}$. For the first part, we apply Lemma \ref{lm:SCS} and note that $h_{l}^{-1}\|\psi_{l,j}\|\eqsim \|\curl \psi_{l,j}\|$ to get
\begin{align*}
\sum_{k=1}^J\sum_{i=1}^{N_k} \sum_{l>k}\sum_{j=1}^{N_l}(u_{k,i}, v_{l,j}) &= \sum_{k=1}^J\sum_{i=1}^{N_k} \sum_{l>k}\sum_{j=1}^{N_l}(\curl \phi_{k,i}, \curl \psi_{l,j})\\
& \leq \sum_{k=1}^J\sum_{i=1}^{N_k} \sum_{l>k}\sum_{j=1}^{N_l} \left (\frac{1}{2}\right )^{l-k}\|\curl \phi_{k,i}\|\|\curl \psi_{l,j}\|\\
&\lesssim \left (\sum_{k=1}^J\sum_{i=1}^{N_k}\|u_{k,i}\|^2 \right )^{1/2} \left (\sum_{l=1}^J\sum_{j=1}^{N_l}\|v_{l,j}\|^2\right )^{1/2}.
\end{align*}
For the second part, we use the finite overlapping property of finite element spaces. Namely, in the $k$th level, the index set $n_k(i) = \{j\in \{1, \ldots, N_k\}, \Omega_{k,i}\cap \Omega_{k,j} \neq \emptyset \}$ is finite. Then
\begin{align*}
\sum_{k=1}^J\sum_{i=1}^{N_k} \sum_{j>i}^{N_k}(u_{k,i}, v_{k,j}) &=\sum_{k=1}^J\sum_{i=1}^{N_k} \sum_{j\in n_k(i)}(u_{k,i}, v_{k,j})\\
&\lesssim \left (\sum_{k=1}^J\sum_{i=1}^{N_k}\|u_{k,i}\|^2 \right )^{1/2} \left (\sum_{l=1}^J\sum_{j=1}^{N_l}\|v_{l,j}\|^2\right )^{1/2}.
\end{align*}
\end{proof}

\subsection{Three dimensions}
We consider the same problem in three dimensions which is much more difficult than the two dimensional case. The reason is that the previous space $\mcal S$ is an edge element space and a stable multilevel decomposition for $\mcal S$ is non-trivial.

We again consider the lowest order case.  Now $\mcal S$ is the lowest order N\'ed\'elec edge element space~\cite{Nedelec.J1980,Nedelec.J1986} of $H_0(\curl,\Omega) :=\{v\in (L^2(\Omega))^3, \curl v \in (L^2(\Omega))^3, v\times n |_{\partial \Omega}= 0\}$, $\mcal V$ is the lowest order Raviart-Thomas element space of $H_0(\div,\Omega)$, and $\mcal P\subset L^2_0(\Omega)$ is the piecewise constant space. Furthermore let $\mcal U\subset H_0^1(\Omega)$ be the linear finite element space. We have the following exact sequence~\cite{Hiptmair.R2002,Arnold2006}
$$
0 \hookrightarrow \mcal U \stackrel{\grad}{\longrightarrow}\mcal S \stackrel{\curl}{\longrightarrow} \mcal V \stackrel{\div}{\longrightarrow} \mcal P \to 0.
$$

To verify (SD), we need the following discrete regular decomposition for edge elements~\cite{Hiptmair.R;Xu.J2007}. In the sequel, the operator $\Pi^{\curl}_h$ is the canonical interpolation to $\mcal S$: for a smooth enough function $w$, $\Pi^{\curl}_h w \in S$ satisfying $\int_E \Pi^{\curl}_h w\cdot t \dd s = \int_E w\cdot t \dd s$ for all edges $E$ of $\mcal T_h$ where $t$ is a tangential vector of $E$. Similarly $\Pi^{\curl}_k$ is the canonical interpolation to $\mcal S_k$ on mesh $\mcal T_k$ for $k=1,\ldots, J$.

\begin{lemma}[Discrete Regular Decomposition~\cite{Hiptmair.R;Xu.J2007}]\label{th:disregdec}
For every $\phi \in \mcal S$, there exist $\tilde \phi \in \mcal S, w \in \mcal U^3$, and $\psi\in \mcal S\cap \ker(\curl)$ such that
\begin{gather}
\label{dec} \phi =  \tilde \phi + \Pi ^{\curl}_hw + \psi,\; \text{ and }\\
\label{stable} \|h^{-1}\tilde \phi\| + \|w\|_1\lesssim
  \|\curl \phi\|.
\end{gather}
\end{lemma}

In the decomposition \eqref{dec}, $\psi \in \ker(\curl)$ and there is no need to control the norm of $\psi$. The component $\tilde \phi$ is of high frequency and the component $w\in \mcal U^3$ for which a stable multilevel decomposition for the linear finite element can be applied. The following decomposition can be found in~\cite{Xu.J;Chen.L;Nochetto.R2009}. 

\begin{lemma}\label{th:curldec}
For every $\phi \in \mcal S$, there exist $\tilde \phi\in \mcal S$, $w_k\in \mcal U_k^3$, and $\psi \in \mcal S\cap \ker(\curl)$ such that
\begin{gather}
\label{muldeccurl}
\phi =  {\tilde \phi} + \sum _{k=1}^J \Pi_k^{\curl}w_k + \psi,\quad \text{and }\\
\label{mulstablecurl}
\|h^{-1}\tilde \phi\|^2 + \sum _{k=1}^Jh_k^{-1} \|w_k\|^2\lesssim \|\curl \phi\|^2.
\end{gather}
\end{lemma}

\begin{theorem}
For every $v\in \mcal V\ \cap \ \ker(\div)$, there exists a decomposition $v = \sum_{k=1}^J\sum_{i=1}^{N_k} v_{k,i}$ such that 
\begin{equation}\label{muldec}
\sum_{k=1}^J\sum_{i=1}^{N_k} \|v_{k,i}\|^2\lesssim \|v\|^2.
\end{equation}
\end{theorem}
\begin{proof}
For $v\in \mcal V\cap \ \ker(\div)$, there exists $\phi\in \mcal S$ such that $v = \curl \phi$. We then apply Lemma \ref{th:curldec} to obtain a decomposition of $\phi$ in the form of \eqref{muldeccurl}. We can write the first two terms in \eqref{muldeccurl} into multilevel basis decomposition, i.e.,
\begin{equation}\label{3Ddec}
{\tilde \phi} + \sum _{k=1}^J \Pi_k^{\curl}w_k = \sum_{k=1}^J\sum_{i=1}^{N_k} \phi_{k,i}.
\end{equation}
Decomposition of $v$ is obtained by choosing $v_{k,i} = \curl \phi_{k,i}$. The stability \eqref{muldec} is from the inverse inequality, the stability of bases decomposition of edge element spaces in $L^2$-norm, and the stability of the decomposition \eqref{mulstablecurl}:
$$
\sum_{k=1}^J\sum_{i=1}^{N_k} \|v_{k,i}\|^2\lesssim \sum_{k=1}^Jh_k^{-2}\sum_{i=1}^{N_k} \|\phi_{k,i}\|^2\lesssim \|h^{-1}\tilde \phi\|^2 + \sum _{k=1}^Jh_k^{-1} \|w_k\|^2\lesssim \|\curl \phi\|^2. 
$$
\end{proof}
The (SCS) can be proved similarly as in the two dimensional case.

\subsection{Numerical examples} In this subsection we present two numerical examples to support our theory. We perform the numerical experiments using the $i$FEM package \cite{Chen.L2008c}.

We consider four examples on the Darcy equations
$$
K^{-1} u + \nabla p = 0, \quad \div u = f \quad \text{ in } \Omega
$$
with given flux boundary condition $u\cdot n = g$ on $\partial \Omega$. We chose $\Omega = (0,1)^2$. Since we focus on the performance of solvers, we only specify the tensor $K$ used in these examples. 

\begin{itemize}
 \item Example 1. $K$ is the identity $2\times 2$ matrix, i.e., $Id_{2\times 2}$ and the grid is uniform. 
 \item Example 2. The grid is still uniform but the tensor is non-diagonal
 $$
 K = 
 \begin{pmatrix}
 1 + 4(x^2+y^2) & 3 xy\\
 3xy &  1 + 11(x^2+y^2)
\end{pmatrix}.
 $$ 
This is the Example 5.2 considered in \cite{Rusten1992}. The spectrum of $K$ is in $[1,25]$ and thus contains certain anisotropy. 

 \item Example 3. The tensor $K = a(x)Id_{2\times 2}$ with piecewise constant $a(x)$ on the initial $4\times 4$ uniform partition of $\Omega$. The scalar function $a = 10^{-p}$ where $p$ is a random integer such that $0\leq p \leq 5$.
 
 \item Example 4. The same tensor in Example 3, except the initial grid is distorted; see Fig. \ref{fig:mesh} (b). The interior grid points are randomly perturbed by up to $40\%$ of the mesh size $h=1/4$. Example 3 and 4 are two dimensional version of the example used in \cite{Wilson2009}.
\end{itemize}

\begin{figure}[htbp]
\label{fig:mesh}
\subfigure[The uniform mesh with $h=1/4$]{
\begin{minipage}[t]{0.5\linewidth}
\centering
\includegraphics*[width=3cm]{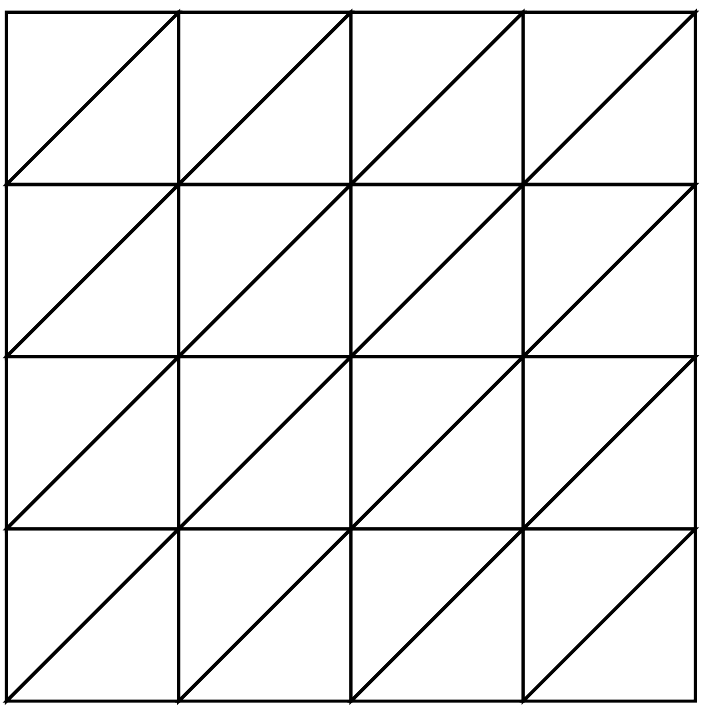}
\end{minipage}}
\subfigure[A distorted mesh]
{\begin{minipage}[t]{0.5\linewidth}
\centering
\includegraphics*[width=3cm]{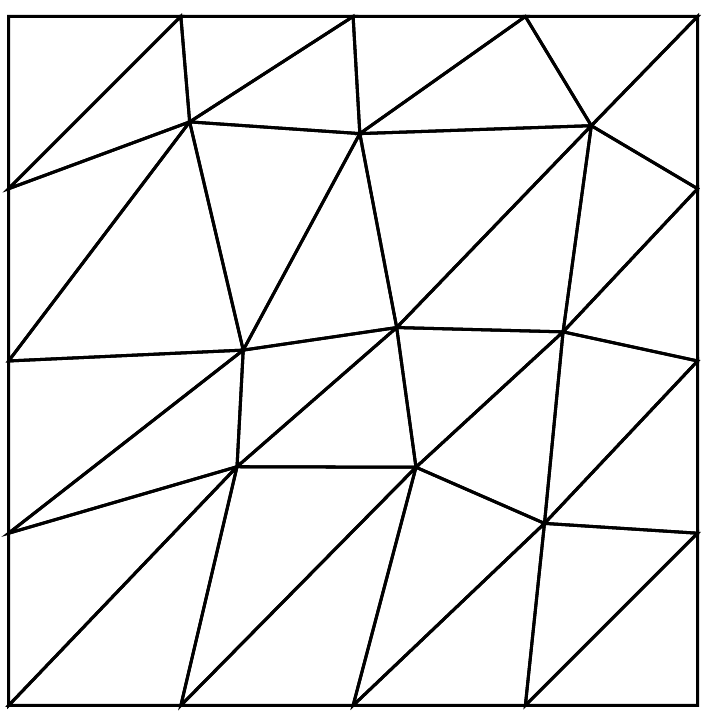}
\end{minipage}}
\caption{The initial mesh of Example 1 - 3 is the uniform mesh in (a) and the initial mesh of Example 4 is a distorted mesh in (b).}
\end{figure}

We discretize the Darcy equations using the lowest order $RT$ element and apply the SSO method with the decomposition \eqref{RTdec}. We implement SSO in a V-cycle formulation and perform only one pre-smoothing and one post-smoothing. The smoother is an overlapping multiplicative Schwarz smoother requiring solving a small saddle point system in the patch of each vertex. The local problem is solved exactly as the dimension of the local problem is small and admit a very efficient direct solver described below. Let $n_i$ denote the number of edges connected to the $i$-th vertex patch. If we orientated these interior edges with normal direction counterclockwise, then locally the divergence free basis is represented by the constant vector $(1,\ldots, 1)_{n_i\times 1}^T$. Note that the local mass matrix $M$ is tridiagonal, given a residual vector, the local problem can be solved in $4n_i$ addition and one division (no multiplication required as the divergence basis corresponds to a constant vector). Let $N$ be the number of interior vertices. The total cost of the local solver is thus $4\sum_{i=1}^Nn_i$. In average $n_i \approx 6$ and thus the cost is around $24 N$. On the other hand, the size of the saddle point system is the number of interior edges plus the number of triangles, which is around $5N$, and the number of non zeros of this matrix is around $21N$. A matrix-vector product thus requires $21N$ multiplication which is way costly than the $24N$ addition needed in the local solvers. Similar calculation holds for the 3D local problem with different constant. We conclude that the dominated cost of the smoother will be the evaluation of the residual and one step of the smoother requires just one matrix-vector product. 

We stop the iteration when an approximated relative error in the energy norm is less than or equal to $10^{-8}$. Let $r$ be the current residual of the iterate $u$ and $Br$ is the correction obtained by one V-cycle. Then we use the error formulae $\sqrt{(Br,r)/(u,f)}$ which is better than using the relative residual error as $B\approx A^{-1}$. We report iteration steps of V-cycle required for the four examples. We did not include the CPU time since it depends on the implementation and testing environment: the programming language, optimization of codes, and the hardware (memory and cache), etc. The operation count we did before indicates that our method can be implemented very efficiently.

\begin{table}[htdp]
\caption{Iteration steps of V-cycle multigrid for the saddle point system with $1$ pre-smoothing and $1$ post-smoothing step. Stopping criterion  is the approximated relative error is less than $10^{-8}$.}
\begin{center}
\begin{tabular}{cccccc}
\hline \hline
$h$ & size & Ex 1 & Ex 2 & Ex 3 & Ex 4\\
\hline
1/8 & 336 & 10 & 13 & 7 & 20\\
\hline
1/16 & 1,312  & 11 & 15 & 10 & 19\\
\hline
1/32 & 5,184 & 11 & 16 & 13 & 25\\
\hline
1/64 & 20,608 & 11 & 16 & 13 & 25\\
\hline \hline
\end{tabular}
\end{center}
\label{tab:iteration}
\end{table}%

Based on the numerical results in Table \ref{tab:iteration}, we conclude that our multigrid method is convergent uniformly to the mesh size and pretty robust to the variation of the tensor $K$ and the distortion of meshes. 

A popular Uzawa type preconditioned conjugate gradient (PCG) method for solving the Schur complement $BM^{-1}B^T$ equation requires the evaluation of $M^{-1}$ (the so-called inner iteration) for each PCG iteration (the so-called outer iteration) and an effective preconditioner for the Schur complement. As noticed in  \cite{Bramble.J;Pasciak.J1988a,Rusten1992}, the inner iteration of computing $M^{-1}$ should be very accurate and thus the overall inner-outer iteration process is costly. And in  \cite{Wilson2009}, it is shown that preconditioners for $BM^{-1}B^T$ should be tuned to the variation of the tensor and the distortion of the mesh. Better preconditioned iterative methods have been developed in~\cite{Bramble.J;Pasciak.J1988a,Rusten1992}.

\section{Application to Non-conforming Methods}
In this section, we use the equivalence between non-conforming methods and mixed methods to develop a V-cycle multigrid method for non-conforming methods and prove its uniform convergence. The two ingredients of our new multigrid method for non-conforming methods are: the overlapping Schwarz smoothers, and inter-grid transfer operators through the nested flux spaces. 


Again we consider the Poisson equation with Neumann boundary condition $-\Delta p = f$ in $\Omega$ with $\partial _{n}p|_{\partial \Omega} = 0$. Based on a triangulation $\mcal T_h$ of $\Omega$, the Crouzeix-Raviart (CR) non-conforming finite element space \cite{CrouzeixRaviart1973} is defined as follows
\begin{equation*}
\Lambda_h = \{\lambda |_{T}\in \mcal P_1(T), \forall\,   T\in \mcal T_h, \int _{E}\lambda \dd s \text{ is continuous for all sides }  E \text{ of } \mcal T_h\}.
\end{equation*}

The space $\Lambda_h$ is not a subspace of $H^1(\Omega)$ due to the loss of  continuity across the sides of elements. An elementwise gradient operator $\nabla _h$ is defined as
$$
(\nabla _h \lambda) |_{T} := \nabla (\lambda |_{T})\quad \forall\,  T\in \mcal T_h,
$$
and the bilinear form is defined as
$$
(\nabla _h \lambda, \nabla _h \mu ) := \sum _{T\in \mcal T_h}\int _{T} \nabla_h \lambda \cdot \nabla_h \mu\, \dx \quad \text{ for all } \lambda, \mu \in \Lambda_h.
$$
The CR non-conforming finite element discretization is as follows: Given an $f\in L^2(\Omega)$, find $\lambda \in \Lambda_h\cap L_0^2(\Omega)$ such that
\begin{equation}\label{cr}
(\nabla _h \lambda, \nabla _h \mu ) = (f, \mu), \quad \text{for all } \mu \in \Lambda_h.
\end{equation}


Let $u\in \mcal V$ be the mixed finite element approximation of the flux using the lowest order RT element. It is well known that~\cite{Marini.L1985}, for every $T\in \mcal T_h$,
\begin{equation}\label{localformula}
u |_{T} = -\nabla _h\lambda |_{T} + \frac{1}{d} f_{T}(x-x_{T}), \quad \forall\,  x\in T,
\end{equation}
where $x_{T}$ is the barycenter of the triangle $T$ and $f_T$ is the average of $f$ over $T$. Throughout this section we shall always consider a piecewise constant function $f$. We always denote by $u$ the solution to \eqref{mixedpoisson0} and by $\lambda$ the solution to \eqref{cr}. 

We note that such equivalence has been used to design multigrid methods for mixed methods with the help of non-conforming methods~\cite{Brenner.S1992,Chen.Z1996}. We are exploiting this equivalence in the other way around. 

Based on a sequence of hierarchy meshes, we will have a sequence of spaces $\Lambda_1, \Lambda_2 ,\ldots ,$ $\Lambda_J = \Lambda_h$. We shall develop a V-cycle multigrid method for solving the equation \eqref{cr} on the finest level. The notorious difficulty is the non-nestedness of hierarchies of non-conforming finite element spaces. Robust inter-grid operators (restriction and prolongation operators) should be designed carefully~\cite{Brenner.S1989,Braess.D;Verfurth.R1990,Oswald.P1997,Chen.Z;Oswald.P1998,Brenner.S1999,Kang.K;Lee.S1999,Koster.M;Ouazzi.A;Schieweck.F;Turek.S2012}. The existing convergence proof of multigrid methods for non-conforming methods~\cite{Brenner.S1989,Brenner.S1999,Oswald.P1997,Brenner2003a,Oswald.P2008} do not cover V-cycles with few smoothing steps but for multigrid cycles with sufficiently many smoothing steps. We shall design a V-cycle multigrid method for CR element and prove its convergence even for only one smoothing step.

Essentially our method is just a different interpretation of the SSO method applied to the mixed finite element discretization. 
%
%
Therefore during the iteration, we always keep two quantities  $( u^k, \lambda^k )$ which is the $k$th iteration of $(u,\lambda)$. 

The smoother in the finest level is an overlapping Schwarz smoother with Neumann boundary condition. It consists of solving a local problem $-\Delta \lambda^{k}_i = f$ in $\Omega_i$ with Neumann boundary condition $\partial _{n}\lambda^{k}_i|_{\partial \Omega_i} = u^k_{i-1}\cdot n$, where $\Omega_i$ is the patch of the $i$th vertex. Here we loop over the vertex $i=1,\ldots, N$ of $\mcal T_h$ and use subscript $i$ to denote the iteration at the $i$th vertex. We set $u^k_0 = u^k$ and $u^{k+1} = u^k_{N}$. Once $\lambda^{k}_i$ is computed, it will be used to update the flux $u^k_i$ by the relation \eqref{localformula} since the relation holds for the local problem as well. To begin with, we need to compute flux $u_*$ on the finest level such that the local Neumann problem is well defined, i.e, the source $f$ is compatible with the prescribed boundary flux. Such flux $u_*$ can be found by a V-cycle multigrid iteration similar to the procedure for non-homogenous  constraint discussed before. 

In the implementation level, the matrix of  local problems can be obtained by extracting sub-matrices of the global one. The right-hand side is the corresponding components of $f$ plus the contribution from the boundary condition. 
It is the degree of freedom $\int _E u\cdot n$ that enters the computation which can be calculated by the formulae
\begin{equation}\label{unlambda}
\int _E u^k_i\cdot n_E \dd s = \frac{|T|}{d+1}f + \int _E\nabla _h \lambda^k_i\cdot n_e \, \dd s.
\end{equation}
The relation \eqref{unlambda} can be used to eliminate the flux and get a direct updated formulation $\lambda^k_{i-1} \to \lambda^k_{i}$ without recording the flux approximation $u^k_i$. Algebraically it can be realized by matrix multiplication of $\lambda^k_i$. Conceptually it is better to record the flux explicitly.

We then discuss the prolongation from the coarse grid to the fine grid. Since now only two levels are involved, we will follow the convention to use subscript $(\cdot )_H$ for quantities in the coarse grid and $(\cdot)_h$ for that in the fine grid.
In the coarse grid, we will solve a residual equation to be considered in a moment. 
Suppose we have obtained a correction of the flux $e_H$, we prolongate $e_H$ in the RT space in the coarse grid to that in the fine grid and denoted by $I_H^he_H$. Note that although spaces of CR non-conforming elements are non-nested, the RT spaces for flux are, and $I_H^h$ is just the natural inclusion. The correction is applied to the flux $u_h \leftarrow u_h + I_H^he_H$. With the updated flux, we have different boundary conditions for the local problems (the source is always $f$ in the finest level) and the smoother in the finest level can be applied again.

We then discuss in detail the residual equation to be solved in the coarse grid. We first describe the restriction. 
Denoted by $u_h$ the current approximation of flux in the fine grid. The residual equation of the corresponding mixed method in the fine grid is 
\begin{equation}\label{mixedresidual}
\begin{pmatrix}
M & B^T\\
B & O
\end{pmatrix}
\begin{pmatrix}
e_h\\
p_h
\end{pmatrix}
=
\begin{pmatrix}
-Mu_h\\
0
\end{pmatrix}.
\end{equation}
So the restriction operator will apply to the residual $-Mu_h$, i.e., $r_H = -(I_H^h)^TMu_h$. 
On the coarse grid, we will still solve local problems on vertex patches. We start with the zero initial guess of flux $e_H$, i.e., $e_{H,0} = 0$, and solve local problems to update $e_{H,i}$ patch-wise for $i=1,\ldots, N_H$. The updated flux $e_{H,i}$ in patch $\Omega_{H,i}$ will provide a boundary condition for the next patch $\Omega_{H,i+1}$. To use the non-conforming formulation, we need to figure out the source data for each local problem. This can be done as follows. Let $r_H^i$ be the restriction of $r_H$ to the $i$th patch $\Omega_{H,i}$, and let $M_H^i$ be the corresponding mass matrix. Then the source for the local problem on $\Omega_{H,i}$ will be given by $\delta f_{H,i} = -\div M_{H,i}^{-1}r_H^i$ which is piecewise constant on $\Omega_{H,i}$. The inverse $M_{H,i}^{-1}$ can be computed efficiently since $M_{H,i}$ is tri-diagonal. Now we can solve the non-conforming discretization of the problem $-\Delta_H \lambda_{H,i} = \delta f_H$ with Neumann boundary condition $\partial_n \lambda_{H,i}|_{\partial \Omega_{H,i}} = e_{H,i-1}\cdot n$ and use $\lambda_{H,i}$ to update the flux correction $e_{H,i}$. Again such procedure can be implemented as one matrix multiplication which leads to a non-trivial restriction matrix.

As usual, a V-cycle multigrid method is obtained by applying the above two-level method recursively to the coarse grid problem.

Convergence of this multigrid algorithm is striaghtforward since it is just a different way to compute the same solution of the mixed formulation for each local problem. The quantity $2(E(u^k) - E(u)) = \| u - u^k\|^2 = \| \nabla _h \lambda - \nabla _h \lambda^k\|^2$ since the relation \eqref{localformula} always holds during the iteration. 

The same algorithm and convergence proof can be applied to other non-conforming methods, e.g., hybridized discontinuous Galerkin (HDG) methods~\cite{Arnold.D;Brezzi.F1985,Cockburn.B;Gopalakrishnan.J2004,CockburnGopalakrishnanLazarov2009} and weak Galerkin (WG) method~\cite{WangYe2012,MuWangYe2012polygon}, which are equivalent to the mixed methods. The only difference is the relation of $\lambda$ and the flux $u$. For example, for WG, we can simply use the following formulae to update the flux: $u = \nabla_w \lambda,$ where $\nabla_w$ is the weak gradient operator.

We thus have obtained a V-cycle multigrid method for non-conforming finite elements and have proved the uniform convergent with even one smoothing step. Such results are very rare in literature and a recent work on a multigrid method for HDG methods with only one smoothing step can be found in~\cite{Hall2013}.

\section{Application to Stokes Equations}
In this section, we apply our approach to designing a multigrid method for a discrete Stokes system in two dimensions and prove its uniform convergence. 

Let $\Omega$ be a polygon and triangulated into a quasi-uniform mesh $\mcal T_h$ with mesh size $h$. Again we assume that there exists  a sequence of meshes $\mcal T_1, \mcal T_2, \ldots, \mcal T_J = \mcal T_h$. The triangulation $\mcal T_1$ is a shape regular triangulation of $\Omega$ and $\mcal T_{k+1}$ is obtained by dividing each triangle  in $\mcal T_{k}$ into four congruent small triangles. We further assume triangulations contain no singular vertex defined in~\cite{Scott.L;Vogelius.M1985}.

Consider Stokes equations $$-\Delta u + \nabla p = f, \quad \div u = 0$$ with Dirichlet boundary condition $u|_{\partial \Omega}=0$. The homogenous boundary condition is not essential. As discussed before, the non-homogenous boundary condition will lead to a non-homogenous constraint and can be eliminated by one V-cycle or by a fast Poisson solver. 

We shall use exact divergence free elements and assume $\mcal K_i = \ker(\div)\cap \mcal V_i, i=1, \ldots, J$ are nested, i.e.,
$$
\mcal K_1 \subset \mcal K_2 \subset \ldots \subset \mcal K_J = \mcal K.
$$
Examples of such Stokes elements include Scott-Vogelius elements~\cite{Scott.L;Vogelius.M1985} for which the assumption that all triangulations contain non-singular vertex is needed. We chose $\mcal V \subset (H_0^1(\Omega))^2$ and $\mcal P \subset L^2_0(\Omega)$ as Scott-Vogelius elements~\cite{Scott.L;Vogelius.M1985}. For this problem, the $A$-norm is the $H^1$ semi-norm $|\cdot|_1 = \|\nabla(\cdot)\|$ which is a norm on $H_0^1$. 

Let $\mcal U$ be the $C^1$ finite element space on $\mcal T$ such that $\curl \mcal U = \mcal K$. Namely we have the so-called Stokes complex:
\begin{equation}\label{2DStokesexact}
\mcal U \stackrel{\curl}{\longrightarrow} \mcal V \stackrel{\div}{\longrightarrow} \mcal P \to 0.
\end{equation}
Similar exact sequence exists in each level $k = 1,2, \ldots, J$.

We further decompose each $\mcal K_k$ into subspaces associated to vertices. For a vertex $x_{k,j}\in \mcal T_k$, we denote by $\Omega_{k,j}$ the patch of $x_{k,j}$, i.e., union of all triangles containing $x_{k,j}$. Let $\mcal U_{k,j} = C_0^1(\Omega_{k,j})\cap \ \mcal U_{k}$ and $\mcal V_{k,j}= (H_0^1(\Omega_{k,j}))^2\cap \ \mcal V_k$ be the subspaces spanned by all basis functions with support in $\Omega_{k,j}$ and set $\mcal K_{i,j} = \mcal V_{i,j}\cap \ker (\div)$. By the construction of Scott-Vogeligus element, $\mcal K_{k,j} = \curl \mcal U_{k,j}$. The final decomposition is
$$
\mcal V = \sum_{k=1}^{J}\sum_{j=1}^{N_k}\mcal V_{k,j}, \text{ and }
\mcal K = \sum_{k=1}^{J}\sum_{j=1}^{N_k}\mcal K_{k,j}.
$$
In the correction form, the smoother applied to this decomposition is equivalent to solving a local Stokes problem with force $f - Av_i$ in the subdomain surrounding of a vertex with zero Dirichlet boundary condition on $\partial \Omega_{k,j}$. In the update form, it is solving the original Stokes problem with force $f$ but the boundary condition on $\partial \Omega_{k,j}$ is given by the current approximation of the velocity. Now the local problem is of considerable size (around a $200\times 200$ saddle point system) and the inexact solver using diagonal matrix of $A$ can reduce it to a SPD problem of smaller size (around $60\times 60$). 

\begin{remark}\rm
We are aware that more effective block preconditioners for the Stokes equations are available~\cite{Benzi.M;Golub.G;Liesen.J2005,Mardal2011a}. Multigrid method based on solving local problems is of more theoretic value since multigrid convergence theory for Stokes equations with the partial regularity assumption and/or for V-cycle method with few smoothing steps is rare. \qed
\end{remark}

We define $Q_k: \mcal U \to \mcal U_k$ the $L^2$ projection, for $k=1,2,\ldots, J$.  
For $v\in \mcal K$, since $\div v = 0$, we can find a unique $\phi\in \mcal U$ such that $v = \curl \phi$. We then define $\Pi_k v = \curl Q_k\phi$. It is easy to show that $\Pi_l\Pi_k = \Pi_l$ for $l\leq k$ due to the nestedness of spaces.

We document the stability and error estimate of $\Pi_k$ in the following lemmas.
\begin{lemma}\label{lm:Pi}
The operator $\Pi_k$ is stable in $L^2$ norm and
\begin{equation}\label{eq:Jackson}
\|v-\Pi_kv\| \lesssim h_k|v|_{1}, \quad \text{for all } v\in \mcal K.
\end{equation}
\end{lemma}
\begin{proof}
It is well known that $Q_k$ is stable in both $L^2$-norm and $\curl$-norm on quasi-uniform meshes. Consequently $\Pi_k = \curl Q_k$ is stable in $L^2$-norm. It is obvious that $\Pi_k v = v$ for all $v\in \mcal K_k$. Therefore
%
$v-\Pi_kv = (I-\Pi_k)(v - v_k)$ for any $v_k\in \mcal K_k$ and consequently
$$
\|v-\Pi_kv\| \lesssim \inf_{v_k\in \mcal K_k}\|v-v_k\| = \inf_{\phi_k\in \mcal U_k}\| \curl \phi -\curl \phi_k \|\lesssim h_k|\phi|_{2} = h_k|v|_{1}.
$$
\end{proof}

\begin{lemma}
The operator $\Pi_k$ is stable in $H^{\sigma}$-norm for $\sigma \in [0,1/2)$, i.e.,
\begin{equation}\label{Pisigma}
\|\Pi_k v\|_{\sigma}\lesssim \|v\|_{\sigma}, \quad\text{ for all }v\in \mcal K.
\end{equation}
\end{lemma}
\begin{proof}
For $\sigma = 0$, i.e., the stability of $\Pi_k$ in $L^2$-norm has been proved in Lemma \ref{lm:Pi}. For $v\in \mcal K$, we define $\bar v$ as the piecewise constant approximation of $v$ defined by $\int_{T}\bar v = \int _Tv$ for all $T\in \mcal T_k$. Obviously $\|v - \bar v \|\lesssim h_k|v|_1$.

We prove the stability of $\Pi_k$ in $H^1$-norm as follows:
\begin{align*}
|\Pi_k v|_1 = |\Pi_k v - \bar v |_1 \lesssim h_k^{-1}\|\Pi_k v - \bar v \| \lesssim h_k^{-1}\left (\|\Pi_k v - v \| + \|v - \bar v \|\right )\lesssim |v|_1.
\end{align*}
In the last step, we have used the approximation property of $\Pi_k$; c.f. Lemma \ref{lm:Pi}.

By the interpolation of divergence free spaces, c.f. Proposition 3.7 in~ \cite{Wendland.H2009}, we obtain the desired inequality \eqref{Pisigma}.

\end{proof}

Following Xu~\cite{Xu1992}, we can obtain the following stable decomposition. For completeness, we include a proof here.
\begin{theorem}\label{th:stokesnorm}
The decomposition $v = \sum_{k=1}^J (\Pi_k - \Pi_{k-1}) v$ is stable in $A$-norm, i.e.,
\begin{equation}\label{Qknorm}
\sum _{k=1}^J |(\Pi_k - \Pi_{k-1}) v|_1^2 \lesssim |v|_{1}^2, \quad\text{ for all } v\in \mcal K.
\end{equation}
\end{theorem}
\begin{proof}
Let $P_i: \mcal K\to \mcal K_i$ be the projection in $A$-inner product. Then by the duality argument, c.f. Theorem 6.9 in~\cite{Chen2015b}, we have the following error estimate, for some $\alpha \in (1/2,1]$, 
\begin{equation}\label{Pialpha}
\|v - P_i v\|_{1-\alpha}\lesssim h_i^{\alpha}\|v\|_1, \text{ for all }v\in H_0^1(\Omega).
\end{equation}
Let $\tilde{\Pi}_k = \Pi_k - \Pi_{k-1}$, and $v_i=(P_i-P_{i-1})v$ for $i=1,2,\cdots, J$ with notation $P_0 = 0$. Using Cauchy-Swarchz inequality, it holds
\begin{align*}
\sum_{k=1}^J\|\nabla(\tilde{\Pi}_kv)\|^2=&\sum_{k=1}^J\sum_{i,j=k}^J\int_{\Omega}\nabla(\tilde{\Pi}_kv_i)
\cdot \nabla(\tilde{\Pi}_kv_j)\,dx \\
=&\sum_{i,j=1}^J\sum_{k=1}^{i\wedge j}\int_{\Omega}\nabla(\tilde{\Pi}_kv_i)\cdot\nabla(\tilde{\Pi}_kv_j)\,dx \\
\leq & \sum_{i,j=1}^J\sum_{k=1}^{i\wedge j}\|\nabla(\tilde{\Pi}_kv_i)\|\|\nabla(\tilde{\Pi}_kv_j)\|,
\end{align*}
where $i\wedge j=\min\{i,j\}$.
According to the inverse inequality, the stability of $\Pi_k$, c.f., \eqref{Pisigma}, and the error estimate of $P_i$ c.f. \eqref{Pialpha}, we have
$$
\|\nabla(\tilde{\Pi}_kv_i)\|\lesssim h_k^{-\alpha}\|\tilde{\Pi}_kv_i\|_{1-\alpha}\lesssim h_k^{-\alpha}\|v_i\|_{1-\alpha}\lesssim h_k^{-\alpha}h_i^{\alpha}|v_i|_{1}.
$$
Combining last two inequalities, we get from the strengthened Cauchy-Swarchz inequality
\begin{align*}
\sum_{k=1}^J\|\nabla (\Pi_k - \Pi_{k-1})v\|^2
\lesssim & 
\sum_{i,j=1}^J\sum_{k=1}^{i\wedge j}h_k^{-2\alpha}h_j^{\alpha}h_i^{\alpha}|v_i|_{1}|v_j|_{1}
\lesssim \sum_{i,j=1}^Jh_{i\wedge j}^{-2\alpha}h_j^{\alpha}h_i^{\alpha}|v_i|_{1}|v_j|_{1} \\
\lesssim & 
\sum_{i,j=1}^J\left (\frac{1}{2}\right )^{\alpha|i-j|}|v_i|_{1}|v_j|_{1} \lesssim \sum_{i=1}^J |v_i|_1^2=|v|_{1}^2.
\end{align*}
\end{proof}

We continue to show the micro-decomposition of the slice $(\Pi_k - \Pi_{k-1})v$ is stable in the energy norm.
\begin{lemma}\label{lm:stokesmicro}
For $v_k = (\Pi_k - \Pi_{k-1})v \in \mcal K_k$, there exists a decomposition $v_k = \sum_{j=1}^{N_k}v_{k,j}$ with $v_{k,j}\in \mcal K_{k,j}$ such that
$$
\sum _{j=1}^{N_j}|v_{k,j}|_1^2 \lesssim |v_k|_1^2.
$$
\end{lemma}
\begin{proof}
Recall that $\phi_k = Q_k \phi$ and $v = \curl \phi$. Let $\phi_{k} - \phi_{k-1} = \sum _{j=1}^{N_k}\psi_{k,j}$ be a decomposition such that $\supp \psi_{k,j}\in \Omega_{k, j}$. Such decomposition can be obtained by partition the basis decomposition. For example, for a basis function associated to an edge, it can be split as half and half to the patch of each vertex of this edge. We then set $v_{k,j} = \curl \psi_{k,j}$ and obtain the decomposition $v_k = \sum_{j=1}^{N_k}v_{k,j}$.  
Then
\begin{align*}
&\sum_{j=1}^{N_k}|v_{k,j}|_1^2 \leq \sum_{j=1}^{N_k}|\psi_{k,j}|_2^2\lesssim \sum_{j=1}^{N_k} h_k^{-4}\|\psi_{k,j}\|^2\lesssim h_k^{-4}\|\phi_k - \phi_{k-1}\|^2 \\
& = h_k^{-4}\|(I-Q_{k-1})(\phi_k - \phi_{k-1})\|^2 \lesssim h_k^{-2} \|\curl(\phi_k - \phi_{k-1})\|^2 = h_k^{-2} \|v_k\|^2.
\end{align*}
We write $v_k = (\Pi_{k} - \Pi_{k-1})v =  (I - \Pi_{k-1})(\Pi_{k} - \Pi_{k-1})v = (I - \Pi_{k-1})v_k$ and use the $L^2$-norm estimate of $\Pi_k$ to conclude
$$
h_k^{-2} \|v_k\|^2\lesssim |v_k|_1^2.
$$
Then the desired inequality follows.
\end{proof}

Combination of Theorem \ref{th:stokesnorm} and Lemma \ref{lm:stokesmicro} leads to the stability of the decomposition $\mcal K = \sum_{k=1}^J\sum_{j=1}^{N_j}\mcal K_{k,j}$ in $A$-norm. 
\begin{theorem}
For every $v\in \mcal K$, there exists a decomposition $v = \sum_{k=1}^J\sum_{j=1}^{N_k}v_{k,j}$ with $v_{k,j}\in \mcal K_{k,j}$ such that
$$
\sum_{k=1}^J\sum _{j=1}^{N_j}|v_{k,j}|_1^2 \lesssim |v|_1^2.
$$
\end{theorem}

The assumption (SCS) is just that for multilevel $H^1$ finite element spaces $\mcal V_1\subset \mcal V_2 \subset \ldots \subset \mcal V_J$ and can be proved similarly as before. Note that since $\mcal V_{k,j}\subset (H_0^1(\Omega_{k,j}))^2$, for functions $v_{k,j}$ in $\mcal V_{k,j}$, the norm equivalence $h_k^{-2}\|v_{k,j}\|\eqsim |v_{k,j}|_1$ holds with an $\mcal O(1)$ constant.

\section{Conclusion and Future Work}
In this paper we have developed a multigrid method for saddle point systems based on a multilevel subspace decomposition of the constraint space $\mcal K$. We have proved the convergence of such method based on the stable decomposition and strengthened Cauchy Schwarz inequality. 
For some mixed finite element discretizations of Poisson, Darcy, and Stokes equations, we have verified SD and SCS assumptions and consequently obtained a multigrid method for the resulting saddle point systems. In a forthcoming work \cite{Chen2015c}, we shall examine a plate bending problem which is a fourth order elliptic equation. The key is to find a underlying exact sequence. 

\bibliographystyle{abbrv}


\begin{thebibliography}{10}

\bibitem{Arnold2006}
D.~Arnold, R.~Falk, and R.~Winther.
\newblock {Finite element exterior calculus, homological techniques, and
  applications}.
\newblock {\em Acta numerica}, 15:1--155, 2006.

\bibitem{Arnold.D;Brezzi.F1985}
D.~N. Arnold and F.~Brezzi.
\newblock {Mixed and nonconforming finite element methods: Implementation,
  postporcessing and error estimates}.
\newblock {\em RAIRO Model Math. Anal. Numer.}, 19:7--32, 1985.

\bibitem{Arnold.D;Falk.R;Winther.R2000}
D.~N. Arnold, R.~S. Falk, and R.~Winther.
\newblock {Multigrid in H(div) and H(curl)}.
\newblock {\em Numer. Math.}, 85:197--218, 2000.

\bibitem{Bank1988}
R.~E. Bank, T.~F. Dupont, and H.~Yserentant.
\newblock {The Hierarchical Basis Multigrid Method}.
\newblock {\em Numer. Math.}, 458:427--458, 1988.

\bibitem{Bank.R;Scott.L1989}
R.~E. Bank and L.~R. Scott.
\newblock {On the conditioning of finite element equations with highly refined
  meshes}.
\newblock {\em SIAM J. Numer. Anal.}, 26:1383--1394, 1989.

\bibitem{Benzi.M;Golub.G;Liesen.J2005}
M.~Benzi, G.~H. Golub, and J.~Liesen.
\newblock {Numerical solution of saddle point problems}.
\newblock {\em Acta Numer.}, 14:1--137, may 2005.

\bibitem{Braess1999}
D.~Braess and W.~Dahmen.
\newblock {A cascadic multigrid algorithm for the Stokes equations}.
\newblock {\em Numer. Math.}, 82:179--191, 1999.

\bibitem{Braess.D;Sarazin.R1997}
D.~Braess and R.~Sarazin.
\newblock {An efficient smoother for the Stokes equation}.
\newblock {\em Appl. Numer. Math.}, 23(1):3--19, feb 1997.

\bibitem{Braess.D;Verfurth.R1990}
D.~Braess and R.~Verf{\"{u}}rth.
\newblock {Multigrid Methods for Nonconforming Finite Element Methods}.
\newblock {\em SIAM J. Numer. Anal.}, 27:979--986, 1990.

\bibitem{Bramble.J;Pasciak.J1988a}
J.~H. Bramble and J.~E. Pasciak.
\newblock {A preconditioning technique for indefinite systems resulting from
  mixed approximations of elliptic problems}.
\newblock {\em Math. Comp.}, 50(181):1--17, 1988.

\bibitem{Brenner.S1989}
S.~C. Brenner.
\newblock {An optimal order multigrid for P1 nonconforming finite elements}.
\newblock {\em Math. Comp.}, 52:1--15, 1989.

\bibitem{Brenner.S1990}
S.~C. Brenner.
\newblock {A nonconforming multigrid method for the stationary Stokes
  equations}.
\newblock {\em Math. Comp.}, 55:411--437, 1990.

\bibitem{Brenner.S1992}
S.~C. Brenner.
\newblock {A Multigrid algorithm for the lowest-order Raviart-Thomas mixed
  triangular finite element method}.
\newblock {\em SIAM J. Numer. Anal.}, 29:647--678, 1992.

\bibitem{Brenner1996a}
S.~C. Brenner.
\newblock {Multigrid methods for parameter dependent problems}.
\newblock {\em RAIRO-M2AN Modelisation Math et Analyse Numerique},
  30(3):265--297, 1996.

\bibitem{Brenner.S1999}
S.~C. Brenner.
\newblock {Convergence of nonconforming multigrid methods without full elliptic
  regularity}.
\newblock {\em Math. Comp.}, 68(225):25--53, 1999.

\bibitem{Brenner2003a}
S.~C. Brenner.
\newblock {Convergence of nonconforming V-cycle and F-cycle multigrid
  algorithms for second order elliptic boundary value problems}.
\newblock {\em Math. Comp.}, 73(247):1041--1066, 2003.

\bibitem{Brenner2014}
S.~C. Brenner, H.~Li, and L.-Y. Sung.
\newblock {Multigrid methods for saddle point problems: Stokes and Lam{\'{e}}
  systems}.
\newblock {\em Numerische Mathematik}, (2):1--24, jan 2014.

\bibitem{Brenner2015}
S.~C. Brenner, D.-S. Oh, and L.-Y. Sung.
\newblock {Multigrid Methods for Saddle Point Problems: Darcy Systems}.
\newblock {\em arXiv}, pages 1--30, 2015.

\bibitem{Brezzi.F;Fortin.M1991}
F.~Brezzi and M.~Fortin.
\newblock {\em {Mixed and hybrid finite element methods}}.
\newblock Springer-Verlag, 1991.

\bibitem{Cai2003}
Z.~Cai, R.~R. Parashkevov, T.~F. Russell, J.~D. Wilson, and X.~Ye.
\newblock {Domain Decomposition for a Mixed Finite Element Method in Three
  Dimensions}.
\newblock {\em SIAM Journal on Numerical Analysis}, 41(1):181--194, jan 2003.

\bibitem{Chen.L2008c}
L.~Chen.
\newblock {iFEM: An Integrated Finite Element Methods Package in MATLAB}.
\newblock {\em Technical Report, University of California at Irvine}, 2009.

\bibitem{Chen.L2009c}
L.~Chen.
\newblock {Deriving the X-Z Identity from Auxiliary Space Method}.
\newblock In Y.~Huang, R.~Kornhuber, O.~Widlund, and J.~Xu, editors, {\em
  Domain Decomposition Methods in Science and Engineering XIX}, pages 309--316.
  Springer Berlin Heidelberg, 2010.

\bibitem{Chen2015b}
L.~Chen.
\newblock {Multigrid methods for saddle point systems using constrained
  smoothers}.
\newblock {\em Computers {\&} Mathematics with Applications}, (2):1--13, 2015.

\bibitem{Chen2015c}
L.~Chen, J.~Hu, and X.~Huang.
\newblock {Multigrid Methods for Hellan-Herrmann-Johnson Mixed Method of
  Kirchhoff Plate Bending Problems}.
\newblock {\em ArXiv}, pages 1--20, 2015.

\bibitem{Chen2013e}
L.~Chen, X.~Hu, M.~Wang, and J.~Xu.
\newblock {A Multigrid Solver based on Distributive Smoother and Defect
  Correction for Oseen Problems}.
\newblock {\em Submitted}, 2013.

\bibitem{Chen.Z1996}
Z.~Chen.
\newblock {Equivalence between and multigrid algorithms for nonconforming and
  mixed methods for second order elliptic problems}.
\newblock {\em East-West Journal of Numerical Mathematics}, 4:1--33, 1996.

\bibitem{Chen.Z;Oswald.P1998}
Z.~Chen and P.~Oswald.
\newblock {Multigrid and multilevel methods for nonconforming rotated Q1
  elements}.
\newblock {\em Math. Comp.}, 67(222):667--693, 1998.

\bibitem{Hall2013}
B.~Cockburn, O.~Dubois, J.~Gopalakrishnan, and S.~Tan.
\newblock {Multigrid for an HDG Method}.
\newblock {\em IMA J. Numer. Anal.}, pages 1--40, 2013.

\bibitem{Cockburn.B;Gopalakrishnan.J2004}
B.~Cockburn and J.~Gopalakrishnan.
\newblock {A Characterization of Hybridized Mixed Methods for Second Order
  Elliptic Problems}.
\newblock {\em SIAM J. Numer. Anal.}, 42(1):283--301, 2004.

\bibitem{CockburnGopalakrishnanLazarov2009}
B.~Cockburn, J.~Gopalakrishnan, and R.~Lazarov.
\newblock {Unified hybridization of discontinuous galerkin, mixed, and
  continuous galerkin methods for second order elliptic problems}.
\newblock {\em SIAM J. Numer. Anal.}, 47(2):1319--1365, 2009.

\bibitem{CrouzeixRaviart1973}
P.~M. Crouzeix and P.~A. Raviart.
\newblock {Conforming and nonconforming finite element methods for solving the
  stationary Stokes equations I}.
\newblock {\em R.A.I.R.O}, 76:3--33, 1973.

\bibitem{Ewing.R;Wang.J1992a}
R.~E. Ewing and J.~Wang.
\newblock {Analysis of the schwarz algorithm for mixed finite elements
  methods}.
\newblock {\em M2AN}, 26:739--756, 1992.

\bibitem{Hiptmair.R2002}
R.~Hiptmair.
\newblock {Finite elements in computational electromagnetism}.
\newblock {\em Acta Numer.}, 11:237--339, 2002.

\bibitem{Hiptmair1999}
R.~Hiptmair and R.~Hoppe.
\newblock {Multilevel methods for mixed finite elements in three dimensions}.
\newblock {\em Numerische Mathematik}, 82:253--279, 1999.

\bibitem{Hiptmair.R;Xu.J2007}
R.~Hiptmair and J.~Xu.
\newblock {Nodal Auxiliary Space Preconditioning in H(curl) and H(div) Spaces}.
\newblock {\em SIAM J. Numer. Anal.}, 45(6):2483--2509, 2007.

\bibitem{Kang.K;Lee.S1999}
K.~S. Kang and S.~Y. Lee.
\newblock {New intergrid transfer operator in multigrid method for
  {\{}P1{\}}-nonconforming finite element method}.
\newblock {\em Appl. Math. Comp..}, 100(2-3):139--149, 1999.

\bibitem{Koster.M;Ouazzi.A;Schieweck.F;Turek.S2012}
M.~Koster, a.~Ouazzi, F.~Schieweck, S.~Turek, P.~Zajac, and M.~K{\"{o}}ster.
\newblock {New robust nonconforming finite elements of higher order}.
\newblock {\em Applied Numerical Mathematics}, 62(3):166--184, mar 2012.

\bibitem{Lee.Y;Wu.J;Xu.J;Zikatanov.L2006}
Y.~J. Lee, J.~Wu, J.~Xu, and L.~Zikatanov.
\newblock {Robust subspace correction methods for nearly singular systems}.
\newblock {\em Mathematical Models and Methods in Applied Sciences},
  17(11):1937--1963, 2006.

\bibitem{Mardal2011a}
K.~Mardal and R.~Winther.
\newblock {Preconditioning discretizations of systems of partial differential
  equations}.
\newblock {\em Numerical Linear Algebra with Applications}, 18:1--40, 2011.

\bibitem{Marini.L1985}
L.~Marini.
\newblock {An Inexpensive Method for the Evaluation of the Solution of the
  Lowest Order Raviart--Thomas Mixed Method}.
\newblock {\em SIAM J. Numer. Anal.}, 22(3):493--496, 1985.

\bibitem{Mathew1993a}
T.~Mathew.
\newblock {Schwarz alternating and iterative refinement methods for mixed
  formulations of elliptic problems, part II: convergence theory}.
\newblock {\em Numerische Mathematik}, 492:469--492, 1993.

\bibitem{Mathew1993}
T.~P. Mathew.
\newblock {Schwarz alternating and iterative refinement methods for mixed
  formulations of elliptic problems, part I: Algorithms and numerical results}.
\newblock {\em Numerische Mathematik}, 65(1):445--468, dec 1993.

\bibitem{MuWangYe2012polygon}
L.~Mu, J.~Wang, and X.~Ye.
\newblock {Weak Galerkin Finite Element Methods on Polytopal Meshes}.
\newblock {\em arXiv preprint arXiv:1204.3655}, 72204:22, apr 2012.

\bibitem{Nedelec.J1980}
J.~C. N{\'{e}}d{\'{e}}lec.
\newblock {Mixed finite elements in R{\^{}}3}.
\newblock {\em Numer. Math.}, 35:315--341, 1980.

\bibitem{Nedelec.J1986}
J.~C. N{\'{e}}d{\'{e}}lec.
\newblock {A new family of mixed finite elements in {\$}R{\^{}}3{\$}}.
\newblock {\em Numer. Math.}, 50:57--81, 1986.

\bibitem{Nochetto.R;Siebert.K;Veeser.A2009}
R.~H. Nochetto, K.~G. Siebert, and A.~Veeser.
\newblock {Theory of adaptive finite element methods: an introduction}.
\newblock In R.~A. DeVore and A.~Kunoth, editors, {\em Multiscale, Nonlinear
  and Adaptive Approximation}. Springer, 2009.

\bibitem{Olshanskii2011}
M.~A. Olshanskii.
\newblock {Multigrid analysis for the time dependent Stokes problem}.
\newblock {\em Mathematics of Computation}, 5718:1--23, 2011.

\bibitem{Oswald.P1997}
P.~Oswald.
\newblock {Intergrid transfer operators and multilevel preconditioners for
  nonconforming discretizations}.
\newblock {\em Appl. Numer. Math.}, 23(1):139--158, 1997.

\bibitem{Oswald.P2008}
P.~Oswald.
\newblock {Optimality of multilevel preconditioning for nonconforming P1 finite
  elements}.
\newblock {\em Numer. Math.}, 111(2):267--291, sep 2008.

\bibitem{Raviart.P;Thomas.J1977}
P.~A. Raviart and J.~Thomas.
\newblock {A mixed finite element method fo 2-nd order elliptic problems}.
\newblock In I.~Galligani and E.~Magenes, editors, {\em Mathematical aspects of
  the Finite Elements Method}, Lectures Notes in Math. 606, pages 292--315.
  Springer, Berlin, 1977.

\bibitem{Rusten1992}
T.~Rusten and R.~Winther.
\newblock {A Preconditioned Iterative Method for Saddlepoint Problems}.
\newblock {\em SIAM Journal on Matrix Analysis and Applications},
  13(3):887--904, jul 1992.

\bibitem{Schoberl1999}
J.~Sch{\"{o}}berl.
\newblock {Multigrid Methods for a Parameter Dependent Problem in Primal
  Variables}.
\newblock {\em Numer. Math.}, 84:1--19, 1999.

\bibitem{Schoberl2003}
J.~Sch{\"{o}}berl and W.~Zulehner.
\newblock {On Schwarz-type Smoothers for Saddle Point Problems}.
\newblock {\em Numerische Mathematik}, 95:377--399, 2003.

\bibitem{Scott.L;Vogelius.M1985}
L.~R. Scott and M.~Vogelius.
\newblock {Norm estimates for a maximal right inverse of the divergence
  operator in spaces of piecewise polynomials}.
\newblock {\em Mathematical Modelling And Numerical Analysis}, 19(1):111--143,
  1985.

\bibitem{Tai.X2003}
X.-C. Tai.
\newblock {Rate of convergence for some constraint decomposition methods for
  nonlinear variational inequalities}.
\newblock {\em Numer. Math.}, 93(4):755--786, 2003.

\bibitem{Tai2001}
X.-C. Tai and J.~Xu.
\newblock {Global and uniform convergence of subspace correction methods for
  some convex optimization problems}.
\newblock {\em Mathematics of Computation}, 71(237):105--125, may 2001.

\bibitem{Uzawa1958}
H.~Uzawa.
\newblock {Iterative methods for concave programming}.
\newblock {\em Studies in linear and nonlinear programming}, 6, 1958.

\bibitem{Vanka1986}
S.~P. Vanka.
\newblock {Block-implicit multigrid solution of Navier-Stokes equations in
  primitive variables}.
\newblock {\em Journal of Computational Physics}, 65(1):138--158, 1986.

\bibitem{Vassilevski.P;Wang.J1992}
P.~S. Vassilevski and J.~Wang.
\newblock {Multilevel iterative methods for mixed finite element
  discretizations of elliptic problems}.
\newblock {\em Numer. Math.}, 63(1):503--520, 1992.

\bibitem{Verfurth1984}
R.~Verf{\"{u}}rth.
\newblock {A multilevel algorithm for mixed problems}.
\newblock {\em SIAM J. Numer. Anal.}, 21(2):264--271, 1984.

\bibitem{WangYe2012}
J.~Wang and X.~Ye.
\newblock {A weak Galerkin finite element method for second-order elliptic
  problems}.
\newblock {\em Journal of Computational and Applied Mathematics}, 241:103--115,
  2013.

\bibitem{Wang2013}
M.~Wang and L.~Chen.
\newblock {Multigrid Methods for the Stokes Equations using Distributive
  Gauss-–Seidel Relaxations based on the Least Squares Commutator}.
\newblock {\em Journal of Scientific Computing}, 56(2):409--431, feb 2013.

\bibitem{Wendland.H2009}
H.~Wendland.
\newblock {Divergence-Free Kernel Methods for Approximating the Stokes
  Problem}.
\newblock {\em SIAM J. Numer. Anal.}, 47(4):3158--3179, 2009.

\bibitem{Wilson2009}
J.~D. Wilson, R.~L. Naff, and T.~F. Russell.
\newblock {Multigrid preconditioned conjugate-gradient solver for mixed
  finite-element method}.
\newblock {\em Computational Geosciences}, 14(2):289--299, aug 2009.

\bibitem{Xu1992}
J.~Xu.
\newblock {Iterative methods by space decomposition and subspace correction}.
\newblock {\em SIAM Rev.}, 34:581--613, 1992.

\bibitem{Xu.J;Chen.L;Nochetto.R2009}
J.~Xu, L.~Chen, and R.~H. Nochetto.
\newblock {Optimal Multilevel Methods for H(grad), H(curl), and H(div) Systems
  on Adaptive and Unstructured Grids}.
\newblock In R.~A. DeVore and A.~Kunoth, editors, {\em Multiscale, Nonlinear
  and Adaptive Approximation}. Springer, 2009.

\bibitem{Xu.J;Zikatanov.L2002}
J.~Xu and L.~Zikatanov.
\newblock {The Method of Alternating Projections and the Method of Subspace
  Corrections in Hilbert Space}.
\newblock {\em J. Amer. Math. Soc.}, 15:573--597, 2002.

\bibitem{Zulehner.W2000}
W.~Zulehner.
\newblock {A class of smoothers for saddle point problems}.
\newblock {\em Computing}, 65(3):227--246, 2000.

\bibitem{Zulehner2011}
W.~Zulehner.
\newblock {Nonstandard Norms and Robust Estimates for Saddle Point Problems}.
\newblock {\em SIAM Journal on Matrix Analysis and Applications},
  32(2):536--560, apr 2011.

\end{thebibliography}
\end{document}